\documentclass[12pt,reqno]{amsart}
\usepackage[margin=1in]{geometry}
\usepackage{amsmath,amssymb,amsthm,graphicx,amsxtra, setspace}
\usepackage[utf8]{inputenc}
\usepackage{mathrsfs}
\usepackage{hyperref}
\usepackage{upgreek}
\usepackage{mathtools}
\usepackage{xcolor}
\allowdisplaybreaks

\usepackage[pagewise]{lineno}

\newtheorem{theorem}{Theorem}[section]
\newtheorem{lemma}[theorem]{Lemma}
\newtheorem{proposition}[theorem]{Proposition}

\newtheorem{definition}[theorem]{Definition}
\newtheorem{example}[theorem]{Example}

\newtheorem{hypothesis}[theorem]{Hypothesis}

\let\originalleft\left
\let\originalright\right
\renewcommand{\left}{\mathopen{}\mathclose\bgroup\originalleft}
\renewcommand{\right}{\aftergroup\egroup\originalright}

\let\emptyset\varnothing

\newcommand{\Tr}{\mathop{\mathrm{Tr}}}
\renewcommand{\d}{\/\mathrm{d}\/}

\def\w{\textbf{W}^{\varepsilon}_{{\theta}^{\varepsilon}}}

\def\e{\varepsilon}

\def\L{\mathbb{L}}
\def\A{\mathrm{A}}

\def\F{\mathrm{F}}
\def\C{\mathrm{C}}
\def\U{\mathrm{U}}
\def\f{\mathbf{f}}

\def\B{\mathrm{B}}
\def\D{\mathrm{D}}
\def\y{\mathbf{y}}

\def\Y{\mathbb{Y}}

\def\E{\mathbb{E}}
\def\X{\mathbb{X}}
\def\x{\mathbf{x}}

\def\g{\mathbf{g}}

\def\z{\mathbf{z}}
\def\v{\mathbf{v}}
\def\V{\mathbb{v}}
\def\w{\mathbf{w}}
\def\W{\mathrm{W}}
\def\G{\mathrm{G}}

\def\M{\mathrm{M}}
\def\N{\mathbb{N}}

\def\V{\mathbb{V}}
\def\wi{\widetilde}

\def\u{\mathrm{U}}

\def\u{\mathbf{u}}
\def\H{\mathbb{H}}

\newcommand{\R}{\mathbb{R}}

\renewcommand{\d}{\/\mathrm{d}\/}

\newcommand{\Addresses}{{% additional braces for segregating \footnotesize
		\footnote{
			\noindent \textsuperscript{1,2}Department of Mathematics, Indian Institute of Technology Roorkee-IIT Roorkee,
			Haridwar Highway, Roorkee, Uttarakhand 247667, INDIA.\par\nopagebreak
			\noindent  \textit{e-mail:} \texttt{Manil T. Mohan: maniltmohan@ma.iitr.ac.in, maniltmohan@gmail.com.}
			
			\textit{e-mail:} \texttt{Kush Kinra: kkinra@ma.iitr.ac.in.}
			
			\noindent \textsuperscript{*}Corresponding author.
			
			\textit{Key words:} Weak pullback mean randon attractor, stochastic convective Brinkman-Forchheimer equations, nonlinear diffusion, Bochner spaces, locally monotone stochastic partial differential equations.
			
			Mathematics Subject Classification (2020): Primary 37L55; Secondary 35B41, 35Q35, 37N10, 35R60.

}}}

\begin{document}
	%	\linenumbers
	
	\title[Weak pullback mean random attractors]{Weak pullback mean random attractors for the stochastic convective Brinkman-Forchheimer equations and locally monotone stochastic partial differential equations
		\Addresses}
	
	\author[K. Kinra and M. T. Mohan]
	{Kush Kinra\textsuperscript{1} and Manil T. Mohan\textsuperscript{2*}}

	\maketitle
	
	\begin{abstract}
		This work is concerned about the asymptotic behavior of the solutions of the two and three dimensional stochastic convective Brinkman-Forchheimer (SCBF) equations$$\d\u-\left[\mu \Delta\u-(\u\cdot\nabla)\u-\alpha\u-\beta|\u|^{r-1}\u-\nabla p\right]\d t=\mathbf{f}\d t+\varepsilon\upsigma(\cdot,\u)\d\W,\ \nabla\cdot\u=0,$$ driven by white noise with nonlinear diffusion terms (for some $\e>0$). We prove the existence and uniqueness of weak pullback mean random attractors for the 2D SCBF equations (for $r\geq1$) as well as 3D SCBF equations (for $r>3$, any $\mu,\beta>0$ and for $r=3$, $2\mu\beta\geq1$) in Bochner spaces, when the diffusion terms are Lipschitz nonlinear functions. Furthermore, we establish the existence of weak pullback mean random attractors for a class of locally monotone stochastic partial differential equations. 
	\end{abstract}

	\section{Introduction} \label{sec1}\setcounter{equation}{0}
	In this work, our main focus is to study the long time behavior (to establish the existence and uniqueness of weak pullback mean random attractors) of the stochastic convective Brinkman-Forchheimer (SCBF) equations as well as a class of locally monotone stochastic partial differential equations (PDEs) driven by white noise with nonlinear diffusion terms. The stochastic Burgers type and semilinear reaction diffusion equations, stochastic 2D Navier-Stokes equations (NSE) and other hydrodynamic models like stochastic magnetohydrodynamic (MHD) equations, the stochastic Boussinesq model for the B\'enard convection, the stochastic 2D magnetic B\'enard problem, stochastic 3D Leray-$\alpha$ model (see \cite{IDC} for more details), stochastic shell model of turbulence, stochastic Ladyzhenskaya model, stochastic power law fluids, etc (cf. \cite{GLS,LR,LR1}, etc and the references therein) are examples of the class of locally monotone PDEs.  
	
	The convective Brinkman-Forchheimer (CBF) equations describe the motion of incompressible fluid flows in a saturated porous medium. The applicability of CBF equations is limited to flows when the velocities are sufficiently high and porosities are not too small, that is, when the Darcy law for a porous medium no longer applies (cf. \cite{PAM}). Let $\mathcal{O}\subset\R^n \ (n=2,3)$ be a bounded domain with boundary $\partial\mathcal{O}$. Let $\u(t , x) \in \R^n$, $p(t,x)\in\R$ represent the velocity field  and  pressure field at time $t$ and position $x$, respectively. We consider the following  nonautonomous stochastic convective Brinkman-Forchheimer equations with $s\in\R$: 
	\begin{equation}\label{1}
	\left\{
	\begin{aligned}
	\d\u+[-\mu \Delta\u+(\u\cdot\nabla)\u+\alpha\u+\beta|\u|^{r-1}\u+\nabla p]\d t&=\mathbf{f}(t)\d t + \varepsilon\upsigma(\cdot,\u)\d\mathrm{W}, \ \text{ in } \ \mathcal{O}\times(s,\infty), \\ \nabla\cdot\u&=0, \ \text{ in } \ \mathcal{O}\times(s,\infty), \\
	\u&=\mathbf{0},\ \ \text{ on } \ \partial\mathcal{O}\times(s,\infty), \\
	\u(s)&=\u_0, \ \text{ in } \ \mathcal{O},
	\end{aligned}
	\right.
	\end{equation}where $\varepsilon$ is positive constant, $\f(\cdot,\cdot)\in\R^n$ represents the external forcing, $\upsigma(\cdot,\cdot)$ is a nonlinear diffusion term, and $\W(\cdot)$ is a two-sided Wiener process of trace class defined on some complete filtered probability space. The constant $\mu>0$ represents the Brinkman coefficient (effective viscosity), the positive constants $\alpha$ and $\beta$ represent the Darcy (permeability of porous medium) and Forchheimer coefficients, respectively. For $\alpha=\beta=0$, the system \eqref{1} reduces to the nonautonomus classical stochastic Navier-Stokes equations (SNSE). The exponent $r\in[1,\infty)$ is called the absorption exponent and  $r=3$ is known as the critical exponent.  The critical homogeneous CBF equations \eqref{1} have the same scaling as the Navier-Stokes equations (NSE) only when $\alpha=0$ (see Proposition 1.1, \cite{HR} and no scale invariance property for other values of $\alpha$ and $r$), which is sometimes referred to as the NSE modified by an absorption term (\cite{SNA}) or the tamed NSE (\cite{MRXZ}). 
	
	For the existence of unique solutions of deterministic CBF equations, the interested readers are referred to see \cite{SNA,FHR,HR,PAM,MTM2}, etc and for its stochastic counterpart see \cite{MTM} for strong solutions and \cite{LHGH1} for martingale solutions. In the stochastic case, for $\mu, \beta>0$, the existence of a pathwise unique strong solution for the two dimensional SCBF equations  for any $r\geq1$ is proved in \cite{MTM}, while  for the three dimensional SCBF equations, it is established for $r>3$.  For the critical case, the same results are established for $2\mu\beta\geq1$, only  (see \cite{MTM} for more details). The monotonicity property of linear and nonlinear operators as well as a stochastic generalization of the Minty-Browder techniques were exploited in the proofs. Similar to the case 3D stochastic NSE, the existence of a unique pathwise strong solution for the 3D SCBF equations for $r\in[1,3)$ is still an open problem. 
	
	Our main interest is on the long term dynamics of problem \eqref{1}, more precisely the existence and uniqueness of weak pullback mean random attractors. The theory of pathwise pullback random attractors was first introduced in \cite{CF,FS}, and thereafter several authors used this theory and proved the existence of random attractors for several SPDEs, for e.g. \cite{BLL,BLW,BGT,BCLLLR,CLR1,CLR2,Crauel1,CDF,FY,GLS,LL,Phan,Wang2,WLYJ,You} etc and the references therein.   The existence of pathwise pullback random attractors for the two and three dimensional SCBF equations in bounded, periodic and unbounded domains  with additive noise is established in \cite{KM,KM1,KM2,KM3}, etc. The existence of pathwise random attractors for locally monotone stochastic partial differential equations with additive L\'evy type noise is examined in \cite{GLS}.
	
	It is observed that all the papers on the pathwise pullback random attractors have imposed either additive noise or multiplicative noise with the nonlinear diffusion term $\upsigma$ that requires either $\upsigma(\cdot,\u)$ be linear in $\u$ or have a very special structure like antisymmetry. To the best of our knowledge,  there is no result available in the literature on the existence of pathwise pullback random attractors for \eqref{1}, when $\upsigma$ is a general Lipschitz nonlinear function, and it is a major problem. In order to overcome this problem, the concept of mean-square as well as weak mean-square random attractors were introduced in  \cite{PEK}, but that too is very restrictive in practice (cf.  \cite{Wang} for a discussion on this).
	
	In order to deal with Lipschitz nonlinear diffusion term $\upsigma(\cdot,\cdot)$, the  concept of weak pullback mean random attractor in the spaces of Bochner integrable functions is introduced in \cite{Wang} and the author investigated the long time behavior of the stochastic reaction-diffusion equations with both nonlinear drift and nonlinear diffusion term. One can think a weak pullback mean random attractor for a mean random dynamical system as a minimal weakly compact and weakly pullback attracting set in a Bochner space (see Definition \ref{def2.9} for more details). Using the concept introduced in \cite{Wang}, the existence of weak mean pullback mean random attractors for the 2D stochastic NSE with nonlinear diffusion terms is proved in \cite{Wang1} and for the non-autonomous $p$-Laplacian equations is established in \cite{GU}. Recently, in \cite{GU1}, the author proved the existence of weak pullback mean random attractors for abstract stochastic evolution equations  such as stochastic reaction-diffusion equations, the stochastic $p$-Laplace equation and stochastic porous media equations. We point out here that the models considered in this work does not fall in the framework of the abstract stochastic evolution equations considered in \cite{GU1}.
	
	In this paper, we prove the existence and uniqueness of weak pullback mean random attractors for the system \eqref{1} in appropriate Bochner spaces. Moreover, we prove the existence of weak pullback mean random attractors for general stochastic evolution equations, which cover several fluid dynamic models and whose coefficient satisfies the locally monotone condition along with some other conditions (see section \ref{sec6} and \cite{IDC,LR,LR1} for more details). We point out that the 2D SCBF equations  \eqref{1} with $r>2$ do not fall in the category of stochastic evolution equations considered in \cite{IDC,LR,LR1}, etc. Thus, we need a different analysis for the SCBF equations even in two dimensions.  As discussed in \cite{Wang1}, we mention here that the existence of invariant weak pullback random attractors for \eqref{1} remains open, as the invariance of weak pullback random attractors requires the weak continuity of the solution operators of stochastic equations in a Bochner space (see \cite{KL}).
	
	The rest of the paper are organized as follows.  In the next section, we discuss about the function spaces, linear and nonlinear operators, hypothesis satisfied by the noise coefficient, the abstract formulation of the system \eqref{1} and global solvability results. We provide the basic definitions and results regarding the existence of weak pullback mean random attractors also in the same section. The section \ref{sec4} is devoted for establishing the existence of weak pullback mean random attractors for the 2D SCBF equations for the absorption exponent $r\in[1,3]$ (Theorem \ref{WPMRA1}). In section \ref{sec5}, we prove the existence of weak pullback mean random attractors for the 2D and 3D SCBF equations for $r>3$ with any $\mu,\beta>0,$ and $r=3$ with $2\mu\beta\geq1$ (Theorem \ref{WPMRA2}). In the final section, we examine the existence of weak pullback mean random attractors for the stochastic evolution equations of the type \eqref{LM_EQ}, which satisfies all the conditions of Hypothesis \ref{LM_F_H}-\ref{LMf_H} (Theorems \ref{WPMRA3} and \ref{WPMRA4}).
	
	\section{Mathematical Formulation}\label{sec3}\setcounter{equation}{0}
	The goal of this section is to present the necessary function spaces and properties of linear and nonlinear operators needed to prove existence and uniqueness of pathwise strong solutions to the system \eqref{1}. 
	\subsection{Function spaces} Let us define   $\mathscr{V}:=\{\u\in\C_0^{\infty}(\mathcal{O},\R^n):\nabla\cdot\u=0\}.$ Let $\H$ be the closure of $\mathscr{V}$ in space $\L^2(\mathcal{O})=\mathrm{L}^2(\mathcal{O};\R^n)$ with the norm  $\|\u\|_{\H}^2:=\int_{\mathcal{O}}|\u(x)|^2\d x,
	$ and inner product $(\u,\v)=\int_{\mathcal{O}}\u(x)\cdot\v(x)\d x,$ for all $\u,\v\in\L^2(\mathcal{O})$, respectively. Let $\V$ be the closure of $\mathscr{V}$ in space $\H_0^1(\mathcal{O})=\mathrm{H}_0^1(\mathcal{O};\R^n)$ with the norm $ \|\u\|_{\V}^2:=\int_{\mathcal{O}}|\nabla\u(x)|^2\d x,
	$ and the inner product $(\!(\u,\v)\!)=(\nabla\u,\nabla\v)=\int_{\mathcal{O}}\nabla\u(x)\cdot\nabla\v(x)\d x,$ for all $\u,\v\in\V$, respectively. Let $\widetilde{\L}^{p}$ be the closure of  $\mathscr{V}$  in space  $\L^p(\mathcal{O})=\mathrm{L}^p(\mathcal{O};\R^n),$ for $p\in(2,\infty)$, with the norm $\|\u\|_{\widetilde{\L}^p}^p=\int_{\mathcal{O}}|\u(x)|^p\d x.$ Let $\langle \cdot,\cdot\rangle $ represent the induced duality between the spaces $\V$  and its dual $\V'$ as well as $\widetilde{\L}^p$ and its dual $\widetilde{\L}^{p'}$, where $\frac{1}{p}+\frac{1}{p'}=1$. Note that $\H$ can be identified with its dual $\H'$. We endow the space $\V\cap\widetilde{\L}^{p}$ with the norm $\|\u\|_{\V}+\|\u\|_{\widetilde{\L}^{p}},$ for $\u\in\V\cap\widetilde{\L}^p$ and its dual $\V'+\widetilde{\L}^{p'}$ with the norm $$\inf\left\{\max\left(\|\v_1\|_{\V'},\|\v_2\|_{\widetilde{\L}^{p'}}\right):\v=\v_1+\v_2, \ \v_1\in\V', \ \v_2\in\widetilde{\L}^{p'}\right\}.$$ Moreover, we have the continuous embedding $\V\cap\widetilde{\L}^p\hookrightarrow\H\hookrightarrow\V'+\widetilde{\L}^{p'}$. In the rest of the paper, we use the notation $\H^2(\mathcal{O}):=\mathrm{H}^2(\mathcal{O};\R^2)$ for the second order Sobolev spaces.
	\subsection{Linear operator}\label{Operator}
	Let $\mathcal{P}: \L^p(\mathcal{O}) \to\wi\L^p,$ $p\in[1,\infty)$, denote the Helmholtz-Hodge projection (cf.  \cite{DFHM}). It is a bounded linear operator and for $p=2$, $\mathcal{P}$ becomes an orthogonal projection (\cite{OAL}). Let us define $$\A\u:=-\mathcal{P}\Delta\u\ \text{ with the domain }\ \D(\A)=\V\cap\H^{2}(\mathcal{O}).$$
	For the bounded domain $\mathcal{O}$, we also have (see subsection 2.2,\cite{MTM})
	\begin{align}\label{poin}
	\lambda_1\|\u\|_{\mathbb{H}}^2\leq\|\u\|_{\mathbb{V}}^2,  \text{ for all }  \u\in\V,
	\end{align}where $\lambda_1$ is the smallest eigenvalue of operator $\A$.
	
	\subsection{Nonlinear operators}
	Let us define the \emph{trilinear form} $b(\cdot,\cdot,\cdot):\V\times\V\times\V\to\R$ by $$b(\u,\v,\w)=\int_{\mathcal{O}}(\u(x)\cdot\nabla)\v(x)\cdot\w(x)\d x=\sum_{i,j=1}^n\int_{\mathcal{O}}\u_i(x)\frac{\partial \v_j(x)}{\partial x_i}\w_j(x)\d x.$$ If $\u, \v$ are such that the linear map $b(\u, \v, \cdot) $ is continuous on $\V$, the corresponding element of $\V'$ is denoted by $\B(\u, \v)$. We also denote  $$\B(\u) = \B(\u, \u)=\mathcal{P}[(\u\cdot\nabla)\u].$$
	An integration by parts yields, for all $\u,\v,\w\in \V$,
	\begin{equation}\label{b0}
	b(\u,\v,\w) =  -b(\u,\w,\v)\ \text{ and }\	b(\u,\v,\v) = 0.
	\end{equation}
	
	Let us now consider the operator $$\mathcal{C}(\u):=\mathcal{P}(|\u|^{r-1}\u).$$ It is immediate that $\langle\mathcal{C}(\u),\u\rangle =\|\u\|_{\widetilde{\L}^{r+1}}^{r+1}$ and the map $\mathcal{C}(\cdot):\widetilde{\L}^{r+1}\to\widetilde{\L}^{\frac{r+1}{r}}$. The following results discuss about the monotonicity properties of linear and nonlinear operators. 
	\begin{lemma}[\cite{MTM}]\label{thm2.2}
		Let $n=2$, $r\in[1,3]$ and $\u_1,\u_2\in\V$. Then,	for the operator $\G(\u)=\mu \A\u+\B(\u)+\beta\mathcal{C}(\u)$, we  have 
		\begin{align}\label{fe2}
		\langle(\G(\u_1)-\G(\u_2),\u_1-\u_2\rangle+ \frac{27}{32\mu ^3}N^4\|\u_1-\u_2\|_{\H}^2\geq 0,
		\end{align}
		for all $\u_2\in{\mathbb{B}}_N$, where ${\mathbb{B}}_N$ is an $\widetilde{\L}^4$-ball of radius $N$, that is,
		$
		{\mathbb{B}}_N:=\big\{\z\in\widetilde{\L}^4:\|\z\|_{\widetilde{\L}^4}\leq N\big\}.
		$
	\end{lemma}
	\begin{lemma}[Theorem 2.2, \cite{MTM}]\label{thm2.4}
		Let $n=2,3,$ $ r> 3$ and $\u_1, \u_2 \in \V\cap\widetilde{\L}^{r+1}.$ Then, we have
		\begin{align}\label{fe4}
		\langle\G(\u_1)-\G(\u_2),\u_1-\u_2\rangle+ \eta_2\|\u_2-\u_2\|_{\H}^2&\geq 0,
		\end{align}
		where $\eta_2=\frac{r-3}{2\mu(r-1)}\left(\frac{2}{\beta\mu (r-1)}\right)^{\frac{2}{r-3}}.$
	\end{lemma}
	\begin{lemma}[Theorem 2.3, \cite{MTM}]\label{thm2.3}
		For $n=r=3$ with $2\beta\mu \geq 1$, the operator $\G(\cdot):\V\to \V'$ is globally monotone, that is, for all $\u_1,\u_2\in\V$, we have 
		\begin{align}\label{fe5}\langle\G(\u_1)-\G(\u_2),\u_1-\u_2\rangle\geq 0.\end{align}
	\end{lemma}
	\subsection{Wiener process}\label{WP}
	Let $(\Omega, \mathscr{F},\{\mathscr{F}_t\}_{t\in\R},\mathbb{P})$ be a complete filtered probability space, where $\{\mathscr{F}_t\}_{t\in\R}$ is an increasing right continuous family of sub-$\sigma$-algebras of $\mathscr{F}$ that contains all $\mathbb{P}$-null sets. 
	
	Let $Q$ be a symmetric non-negative bounded linear trace class operator in $\H$. The stochastic process $\{\W(t):t\in\R\}$ is an $\H$-valued Wiener process with covariance $Q$ if and only if for arbitrary $t$, the process $\W(t)$ can be expressed as $$\W(t)=\sum_{k=1}^{\infty} \sqrt{\mu_k}e_k(x)\beta_k(t),$$ where $\beta_k(t), k\in\N$ are independent one dimensional Brownian motions on $(\Omega,\mathscr{F},\mathbb{P})$ and $\{e_k\}_{k=1}^{\infty}$ are the orthonormal basis functions of $\H$ such that $Qe_k=\mu_ke_k$ (Proposition 4.3, \cite{DZ1}). Let $\H_0=Q^{1/2}\H$ and $\mathcal{L}_2(\H_0,\H)$ be the space of Hilbert-Schmidt operators from $\H_0$ to $\H$ with norm $\|\cdot\|_{\mathcal{L}_2(\H_0,\H)}$ given by 
	$$\|\Psi\|^2_{\mathcal{L}_2(\H_0,\H)}=\Tr(\Psi Q\Psi^*),\ \text{ for all } \ \Psi \in \mathcal{L}_2(\H_0,\H),$$ where $\Psi^*$ is the adjoint of the operator of $\Psi$. 
	\subsection{Abstract formulation}
	On taking projection $\mathcal{P}$ onto the first equation in \eqref{1}, we obtain 
	\begin{equation}\label{S-CBF}
	\left\{
	\begin{aligned}
	\d\u(t)+\{\mu \A\u(t)+\B(\u(t))+\alpha\u(t)+\beta \mathcal{C}(\u(t))-\f(t)\}\d t&=\varepsilon\upsigma(t,\u) \d\mathrm{W}(t), \ \ \ t>s, \\ 
	\u(s)&=\u_0,
	\end{aligned}
	\right.
	\end{equation}
	where $\u_0\in \H$, $\W(t)$ is a $Q$-Wiener process defined in $\H$. The noise coefficient satisfies the following Hypothesis:  
	
	\begin{hypothesis}\label{sigmaH}
		The noise coefficient $\upsigma(\cdot,\cdot)$ satisfies the following:
		\begin{itemize}
			\item [(H.1)] The function $\upsigma\in\mathrm{C}(\R\times\V;\mathcal{L}_2(\H_0,\H)).$
			\item[(H.2)] \emph{(Growth condition)} There exists a positive constant $K$ such that for all $\u\in\H$ and $t\in\R$,$$\|\upsigma(t,\u)\|^2_{\mathcal{L}_2(\H_0,\H)}\leq K(1+\|\u\|_{\H}^2).$$
			\item[(H.3)] \emph{(Lipschitz condition)} There exists a positive constant $L$ such that for all $\u,\v\in\H$ and $t\in\R$,$$\|\upsigma(t,\u)-\upsigma(t,\v)\|^2_{\mathcal{L}_2(\H_0,\H)}\leq L\|\u-\v\|_{\H}^2.$$
		\end{itemize}
	\end{hypothesis}
	\begin{example}[\cite{Wang1}]
		Let us discuss an example of the diffusion term $\upsigma,$ which satisfies all the conditions of Hypothesis \ref{sigmaH}. Let $e_0$ be an arbitrary  element of $\H$ with $\|e_0\|_{\H}=1$ and let $\emph{span}\{e_0\}$ be the space spanned by $e_0$. Let $Q$ be the projection operator from $\H$ to $\emph{span}\{e_0\}$. Then $Q$ is a symmetric non-negative bounded linear trace class operator in $\H$.
		
		Let $\beta_0(\cdot)$ be a real-valued Wiener process on $(\Omega, \mathscr{F},\{\mathscr{F}_t\}_{t\in\R},\mathbb{P})$ and $\W(t)=\beta_0(t)e_0.$ Then $\W(t)$ is a $Q$-Wiener process defined on $(\Omega, \mathscr{F},\{\mathscr{F}_t\}_{t\in\R},\mathbb{P})$ taking values in $\H$, and $\H_0=Q^{1/2}\H=\emph{span}\{e_0\}.$ Let $\upsigma_0:\V\to\H$ be the mapping given by $$\upsigma_0(\u)=\u+\emph{sin}\u,\   \text{ for all }\ \u\in\V.$$It is easy to show that there exists two positive constant K and  L such that 
		\begin{align}\label{sigma_0_1}
		\|\upsigma_0(\u)\|^2_{\H}\leq K(1+\|\u\|^2_{\H}), \  \text{ for all }\  \u\in\V,
		\end{align}
		and
		\begin{align}\label{sigma_0_2}
		\|\upsigma_0(\u)-\upsigma_0(\v)\|^2_{\H}\leq L\|\u-\v\|^2_{\H}, \  \text{ for all }\ \u,\v\in\V.
		\end{align}
		Given $\u\in\V$, let $\upsigma:\H_0\to\H$ be defined by 
		\begin{align}
		\upsigma(\u)(\v_0)=(\v_0,e_0)\upsigma_0(\u), \  \text{ for all } \ \v_0\in\H_0.
		\end{align}
		It shows that $\upsigma(\u)\in\mathcal{L}_2(\H_0,\H)$ and $$\|\upsigma(\u)\|^2_{\mathcal{L}_2(\H_0,\H)}=\|\upsigma(\u)(e_0)\|_{\H}^2=\|\upsigma_0(\u)\|^2_{\H}$$ and $$\|\upsigma(\u)-\upsigma(\v)\|^2_{\mathcal{L}_2(\H_0,\H)}=\|\upsigma_0(\u)-\upsigma_0(\v)\|^2_{\H}.$$Hence, It is clear from \eqref{sigma_0_1} and \eqref{sigma_0_2} that $\upsigma$ satisfies all the conditions of Hypothesis \ref{sigmaH}.
	\end{example}

	\subsection{Notations and preliminaries}
	In this section, we provide  some basic definitions and results on the existence of weak pullback mean random attractors for mean random dynamical systems, which have been borrowed from \cite{KL} (see \cite{Wang,Wang1} also). 
	
	Let $\mathrm{X}$ be a Banach space with the norm $\|\cdot\|_{\mathrm{X}}$. Given $p\in(1,\infty)$, let $\mathrm{L}^{p}(\Omega,\mathscr{F};\mathrm{X})$ be the Bochner space consisting of all Bochner integrable functions $\psi:\Omega\to\mathrm{X}$ such that
	\begin{align*}
	\|\psi\|_{\mathrm{L}^{p}(\Omega,\mathscr{F};\mathrm{X})} = \bigg(\int_{\Omega}\|\psi\|^p_{\mathrm{X}}\d \mathbb{P}\bigg)^{1/p}<\infty.
	\end{align*}
	For every $s\in\R,$ the space $\mathrm{L}^{p}(\Omega,\mathscr{F}_s;\mathrm{X})$ is defined similarly.
	
	Let $\mathfrak{D}$ be a collection of some families of nonempty bounded subsets of $\mathrm{L}^{p}(\Omega,\mathscr{F}_s;\mathrm{X})$ parametrized by $s\in\R,$ that is, 
	\begin{align}\label{ClassD}
	\mathfrak{D}=\{\mathrm{D}=\{\mathrm{D}(s)\subseteq\mathrm{L}^{p}(\Omega,\mathscr{F}_{s};\mathrm{X}):\mathrm{D}(s)\neq\emptyset \text{ bounded, } s\in\R\}:\mathrm{D} \text{ satisfies some conditions}\}.
	\end{align}
	Such a collection $\mathfrak{D}$ is called \emph{inclusion-closed} if $\mathrm{D}=\{\mathrm{D}(s):s\in\R\}\in \mathfrak{D}$ implies that every family $\mathrm{\widetilde{D}}=\{\mathrm{\widetilde{D}}(s):\emptyset\neq\mathrm{\widetilde{D}}(s)\subseteq\mathfrak{D}(s),\ \text{ for all } \ s\in\R\}$ also belongs to $\mathfrak{D}$. Let us now give the definition of mean random dynamical systems. 
	
	\begin{definition}\label{def_MRDS}
		A family $\Phi=\{\Phi(t,s):t\in\R^+, s\in \R\}$ of mappings is called a \emph{mean random dynamical system} on $\mathrm{L}^{p}(\Omega,\mathscr{F};\mathrm{X})$ over $(\Omega,\mathscr{F},\{\mathscr{F}_t\}_{t\in\R}, \mathbb{P}),$ if for all $s\in\R$ and $t,t_1,t_2\in \R^+,$ 
		\begin{itemize}
			\item [(i)] $\Phi(t,s)$ maps $\mathrm{L}^{p}(\Omega,\mathscr{F}_{s};\mathrm{X})$ to $\mathrm{L}^{p}(\Omega,\mathscr{F}_{t+s};\mathrm{X});$
			\item [(ii)] $\Phi(0,s)$ is the identity operator on $\mathrm{L}^{p}(\Omega,\mathscr{F};\mathrm{X});$
			\item [(iii)] $\Phi(t_1+t_2,s)=\Phi(t_1, s+t_2)\circ\Phi(t_2,s).$
		\end{itemize}
	\end{definition}
	\begin{definition}
		A family $\mathrm{K}=\{\mathrm{K}(s):s\in\R\}\in\mathfrak{D}$ is called a $\mathfrak{D}$-pullback absorbing set for $\Phi$ on $\mathrm{L}^{p}(\Omega,\mathscr{F};\mathrm{X})$ over $(\Omega, \mathscr{F},\{\mathscr{F}_t\}_{t\in\R},\mathbb{P}),$ if for every $s\in\R$ and $\mathrm{D}\in\mathfrak{D}$, there exists $T=T(s,\mathrm{D})>0$ such that 
		\begin{align*}
		\Phi(t,s-t)(\mathrm{D}(s-t))\subseteq\mathrm{K}(s) , \ \text{ for all }\ t\geq T.
		\end{align*}
		If, in addition, $\mathrm{K}(s)$ is weakly compact nonempty subset of $\mathrm{L}^{p}(\Omega,\mathscr{F}_s;\mathrm{X}),$ for every $s\in\R,$ then $\mathrm{K}=\{\mathrm{K}(s):s\in\R\}$ is called a \emph{weakly compact $\mathfrak{D}$-pullback absorbing set} for $\Phi$.
	\end{definition}
	\begin{definition}
		A family $\mathrm{K}=\{\mathrm{K}(s):s\in\R\}\in\mathfrak{D}$ is called a \emph{$\mathfrak{D}$-pullback weakly attracting set} of $\Phi$ on $\mathrm{L}^{p}(\Omega,\mathscr{F};\mathrm{X})$ over $(\Omega, \mathscr{F},\{\mathscr{F}_t\}_{t\in\R},\mathbb{P}),$ if for every $s\in\R$, $\mathrm{D}\in\mathfrak{D}$ and every weak neighborhood $\mathcal{N}^{w}(\mathrm{K}(s))$ of $\mathrm{K}(s)$ in $\mathrm{L}^{p}(\Omega,\mathscr{F}_s;\mathrm{X})$, there exists $T=T(s,\mathrm{D},\mathcal{N}^w(\mathrm{K}(s)))>0$ such that 
		\begin{align*}
		\Phi(t,s-t)(\mathrm{D}(s-t))\subseteq\mathcal{N}^w(\mathrm{K}(s)), \ \text{ for all }\ t\geq T.
		\end{align*}
	\end{definition}
	\begin{definition}\label{def2.9}
		A family $\mathcal{A}=\{\mathcal{A}(s):s\in\R\}\in\mathfrak{D}$ is called a \emph{weak $\mathfrak{D}$-pullback mean random attractor} for $\Phi$ on  $\mathrm{L}^{p}(\Omega,\mathscr{F};\mathrm{X})$ over $(\Omega, \mathscr{F},\{\mathscr{F}_t\}_{t\in\R},\mathbb{P}),$ if the following conditions \emph{(i)-(iii)} are satisfied:
		\begin{itemize}
			\item [(i)] $\mathcal{A}(s)$ is a weakly compact subset of $\mathrm{L}^{p}(\Omega,\mathscr{F}_s;\mathrm{X})$ for every $s\in\R$.
			\item[(ii)] $\mathcal{A}$ is a $\mathfrak{D}$-pullback weakly attracting set of $\Phi$.
			\item[(iii)] $\mathcal{A}$ is the minimal element of $\mathfrak{D}$ with properties \emph{(i)} and \emph{(ii)}; that is, if $\mathcal{B}=\{\mathcal{B}(s):s\in\R\}\in \mathfrak{D}$ satisfies \emph{(i)} and \emph{(ii)}, then $\mathcal{A}(s)\subseteq\mathcal{B}(s),$ for all $s\in\R$.
		\end{itemize} 
	\end{definition}
	Next, we provide  the result on the existence of weak $\mathfrak{D}$-pullback mean random attractors, which is proved in \cite{Wang}.
	\begin{theorem}[Theorem 2.7, \cite{Wang}]\label{Main-T}
		Suppose $\mathrm{X}$ is a reflexive Banach space and $p\in(1,\infty)$. Let $\mathfrak{D}$ be the inclusion-closed collection of some families of nonempty bounded subsets of $\mathrm{L}^{p}(\Omega,\mathscr{F};\mathrm{X})$ as given by \eqref{ClassD} and let $\Phi$ be a mean random dynamical system on $\mathrm{L}^{p}(\Omega,\mathscr{F};\mathrm{X})$ over $(\Omega,\mathscr{F},\{\mathscr{F}_{t}\}_{t\in\R},\mathbb{P})$. If $\Phi$ has a weakly compact $\mathfrak{D}$-pullback absorbing set $\mathrm{K}\in\mathfrak{D}$ on $\mathrm{L}^{p}(\Omega,\mathscr{F};\mathrm{X})$ over $(\Omega,\mathscr{F},\{\mathscr{F}_t\}_{t\in\R},\mathbb{P})$, then $\Phi$ has a unique weak $\mathfrak{D}$-pullback mean random attractor $\mathcal{A}\in\mathfrak{D}$, which is given by, for each $s\in\R$, 
		\begin{align*}
		\mathcal{A}(s)=\bigcap_{s\geq0}\overline{\bigcup_{t\geq s}\Phi(t,s-t)(\mathrm{K}(s-t))}^w,
		\end{align*}
		where the closure is taken with respect to the weak topology of $\mathrm{L}^{p}(\Omega,\mathscr{F}_s;\mathrm{X})$.
	\end{theorem} 
	As discussed in \cite{Wang1}, we remark that $\mathrm{L}^{p}(\Omega,\mathscr{F};\mathrm{X})$ is not required to be separable in Theorem \ref{Main-T}, and hence the weak topology on bounded subsets of $\mathrm{L}^{p}(\Omega,\mathscr{F};\mathrm{X})$ is not metrizable. As such, Theorem \ref{Main-T} does not follow directly from the attractors theory in a metric space.

	\section{Weak Pullback Mean Random Attractors for \eqref{S-CBF}, $n=2$ and $r\in[1,3]$}\label{sec4}\setcounter{equation}{0}
	This section is devoted to discuss about the mean random dynamical system for the 2D nonautonomous SCBF equations  \eqref{1} with $r\in[1,3]$ over a filtered probability space  and establish the existence and uniqueness of weak pullback mean random attractors for the system \eqref{S-CBF}. Let us first provide the solvability results for the system \eqref{S-CBF}.
	\begin{definition}
		For $r\in[1,3]$, let $s\in\R$ and $\u_0\in \mathrm{L}^4(\Omega, \mathscr{F}_s;\H).$ Then, an $\H$-valued $\{\mathscr{F}_t\}_{t\in\R}$-adapted stochastic process $\{\u(t)\}_{t\in[s,\infty)}$ is called a \emph{strong solution} of \eqref{S-CBF} on $[s,\infty)$ with initial data $\u_0$ if $\u\in\mathrm{C}([s,\infty);\H)\cap\mathrm{L}^2_{\mathrm{loc}}((s,\infty);\V), \  \mathbb{P}$-a.s., and satisfies, for every $t>s$ and $\v\in\V$,
		\begin{align*}
		(\u(t),\v)
		&+ \int_{s}^{t} \langle\mu\A\u(\tau)+ \B\u(\tau)+\alpha\u(\tau)+\beta\mathcal{C}(\u(\tau)),\v\rangle\d \tau \\&= (\u_0,\v)+ \int_{s}^{t}\langle\f(\tau), \v\rangle\d\tau +\varepsilon\int_{s}^{t} (\v, \upsigma(\tau,\u)\d\W(\tau)),
		\end{align*}$\mathbb{P}$-a.s. A strong solution $\u(\cdot)$ to the system \eqref{S-CBF} is called a
		\emph{pathwise  unique strong solution} if
		$\widetilde{\u}(\cdot)$ is an another strong
		solution, then $$\mathbb{P}\big\{\omega\in\Omega:\u(t)=\widetilde{\u}(t),\ \text{ for all }\ t\in[0,T]\big\}=1.$$ 
	\end{definition}
	Under the Hypothesis \ref{sigmaH} on $\upsigma$, we provide the following result on the existence and uniqueness of pathwise strong solutions to the system \eqref{S-CBF}, which is proved in \cite{MTM1}.
	\begin{proposition}[Lemma 3.1, \cite{MTM1}]\label{MRDS}
		For $r\in[1,3]$, let all the conditions of Hypothesis \ref{sigmaH} be satisfied. Then, there exists $\varepsilon_0>0$ such that for every $\varepsilon\in(0,\varepsilon_0), s\in\R$, $\u_0\in \mathrm{L}^4(\Omega,\mathscr{F}_s;\H)$ and $\f\in \mathrm{L}^4_{\mathrm{loc}}(\R;\V')$, system \eqref{S-CBF} has a unique solution $$\u\in\mathrm{L}^4(\Omega; \mathrm{C}([s,s+T];\H))\cap\mathrm{L}^2(\Omega; \mathrm{L}^2((s,s+T);\V))\cap\mathrm{L}^{r+1}(\Omega;\mathrm{L}^{r+1}((s,s+T);\widetilde{\L}^{r+1})),$$ for every $T>0.$ Moreover, 
		\begin{align}\label{MRDS1}
		\mathbb{E}\bigg[\sup_{\tau\in[s,s+T]}\|\u(\tau)\|^4_{\H}&+ \mu\int_{s}^{s+T}\|\u(\tau)\|^2_{\H}\|\u(\tau)\|^2_{\V}\d\tau+\alpha\int_{s}^{s+T}\|\u(\tau)\|^4_{\H}\d\tau\nonumber\\&+\beta\int_{s}^{s+T}\|\u(\tau)\|^2_{\H}\|\u(\tau)\|^{r+1}_{\widetilde{\L}^{r+1}}\d\tau\bigg]\leq \M(1+\mathbb{E}(\|\u_0\|^4_{\H})),
		\end{align} where $\M>0$ is a constant independent of $\u_0$.
	\end{proposition}
	Note that $\u\in\mathrm{C}([s,\infty),\H)$ $\mathbb{P}$-a.s. Therefore, by \eqref{MRDS1} and the Lebesgue dominated convergence theorem we obtain that $\u\in\mathrm{C}([s,\infty),\mathrm{L}^4(\Omega,\mathscr{F};\H)),$ which helps us to define a mean random dynamical system for the system \eqref{S-CBF}. Given $t\in\R^+$ and $s\in\R,$ let $\Phi_1(t,s)$ be a mapping from $\mathrm{L}^4(\Omega,\mathscr{F}_s;\H)$ to $\mathrm{L}^4(\Omega,\mathscr{F}_{s+t};\H)$ defined by $$\Phi_1(t,s)(\u_0)=\u(t+s,s,\u_0),$$ where $\u_0\in \mathrm{L}^4(\Omega,\mathscr{F}_s;\H)$, and $\u(\cdot)$ is the unique strong solution to the system \eqref{S-CBF}. Since the solution of system \eqref{S-CBF} is unique, we get that for every $t_1,t_2\geq0$ and $s\in\R$, $$\Phi_1(t_1+t_2,s)=\Phi_1(t_1,t_2+s)\circ\Phi_1(t_2,s).$$ Consequently, $\Phi_1$ is a mean random dynamical system on $\mathrm{L}^4(\Omega,\mathscr{F};\H)$ over $(\Omega, \mathscr{F},\{\mathscr{F}_t\}_{t\in\R},\mathbb{P})$ in the sense of Definition \ref{def_MRDS}.
	
	Let $\mathcal{B}=\{\mathcal{B}(s)\subseteq\mathrm{L}^4(\Omega,\mathscr{F}_s;\H):s\in\R\}$ be a family of nonempty bounded sets such that \begin{align}\label{classB}
	\lim_{s\to-\infty}e^{\mu\lambda_1 s} \|\mathcal{B}(s)\|^4_{\mathrm{L}^4(\Omega,\mathscr{F}_s;\H)}=0,
	\end{align} where $$\|\mathcal{B}(s)\|_{\mathrm{L}^4(\Omega,\mathscr{F}_s;\H)}=\sup_{\u\in\mathcal{B}(s)}\|\u\|_{\mathrm{L}^4(\Omega,\mathscr{F}_s;\H)}.$$
	Let $\mathfrak{D}$ be the collection of all families of nonempty bounded sets with the property \eqref{classB}:
	\begin{align}
	\mathfrak{D}=\{\mathcal{B}=\{\mathcal{B}\subseteq\mathrm{L}^4(\Omega,\mathscr{F}_s;\H):\mathcal{B}\neq\phi \text{ bounded}, s\in\R\}:\mathcal{B} \text{ satisfies } \eqref{classB}\}.
	\end{align}
	Our aim is to prove the existence and uniqueness of weak $\mathfrak{D}$-pullback mean random attractors of $\Phi_1$ and we require the following assumption on $\f(\cdot)$.
	\begin{hypothesis}\label{f_H}
		The deterministic forcing term $\f$ satisfies
		\begin{align}\label{f_H1}
		\int_{-\infty}^{s} e^{\mu\lambda_1\tau}\|\f(\tau)\|^4_{\V'} \d\tau<\infty,\ \ \text{ for all } s\in\R.
		\end{align}
		From \eqref{f_H1}, we also have 
		\begin{align}\label{f_H2}
		\lim_{s\to-\infty} 	\int_{-\infty}^{s} e^{\mu\lambda_1\tau}\|\f(\tau)\|^4_{\V'} \d\tau=0.
		\end{align}
	\end{hypothesis}
	\subsection{Weak $\mathfrak{D}$-pullback mean random attractors}
	In this subsection, we prove the existence and uniqueness of weak $\mathfrak{D}$-pullback mean random attractors for the system \eqref{S-CBF} in $\mathrm{L}^4(\Omega,\mathscr{F};\H)$ over $(\Omega,\mathscr{F},\{\mathscr{F}_t\}_{t\in\R},\mathbb{P}).$
	\begin{lemma}\label{absorb}
		For $r\in[1,3]$, let us assume that \eqref{f_H1} and all the conditions of Hypothesis \ref{sigmaH} are satisfied. Then, there exists $\varepsilon_0>0$ such that for every $0<\varepsilon\leq\varepsilon_0$ and for every $s\in\R$ and $\mathcal{B}=\{\mathcal{B}(t)\}_{t\in\R}\in \mathfrak{D},$ there exists $T=T(s,\mathcal{B})>0$ such that for all $t\geq T,$ the solution $\u(\cdot)$ of system \eqref{S-CBF} satisfies  
		\begin{align}\label{ue}
		\mathbb{E}\left[\|\u(s,s-t,\u_0)\|_{\H}^4\right]\leq 1 +\frac{\mu\lambda_1}{8} +\frac{4}{\mu^3\lambda_1}e^{-\mu\lambda_1 s}\int_{-\infty}^{s}e^{\mu\lambda_1\tau}\|\f(\tau)\|^4_{\V'}\d\tau,
		\end{align} where $\u_0\in\mathcal{B}(s-t)$.
	\end{lemma}
	\begin{proof}
		Applying the infinite dimensional It\^o formula (cf. \cite{GK1,Me}) to the process $\|\u(\cdot)\|^2_{\H}$, we find
		\begin{align}\label{ue1}
		\d\|\u(\xi)\|_{\H}^2=&\big(-2\mu\|\u(\xi)\|^2_{\V}-2\alpha\|\u(\xi)\|^2_{\H}-2\beta\|\u(\xi)\|^{r+1}_{\widetilde{\L}^{r+1}}+2\langle\f(\xi),\u(\xi)\rangle\nonumber\\&\quad+\varepsilon^2\|\upsigma(\xi,\u(\xi))\|^2_{\mathcal{L}_2(\H_0,\H)}\big)\d\xi +2\varepsilon\big(\u(\xi),\upsigma(\xi,\u(\xi))\d\W(\xi)\big).
		\end{align}
		Again by It\^o's formula and \eqref{ue1}, we get
		\begin{align}\label{ue2}
		\d\|\u(\xi)\|^4_{\H}=&2\|\u(\xi)\|^2_{\H}\big(-2\mu\|\u(\xi)\|^2_{\V}-2\alpha\|\u(\xi)\|^2_{\H}-2\beta\|\u(\xi)\|^{r+1}_{\widetilde{\L}^{r+1}}+2\langle\f(\xi),\u(\xi)\rangle\nonumber\\&\quad+\varepsilon^2\|\upsigma(\xi,\u(\xi))\|^2_{\mathcal{L}_2(\H_0,\H)}\big)\d\xi +4\varepsilon^2\|\upsigma^*(\xi,\u(\xi))\u(\xi)\|^2_{\H_0}\d\xi \nonumber\\&\quad + 4\varepsilon\|\u(\xi)\|^2_{\H}\big(\u(\xi),\upsigma(\xi,\u(\xi))\d\W(\xi)\big),
		\end{align}
		where $\upsigma^*$ is the adjoint of the operator of $\upsigma$. Taking the expectation in \eqref{ue2},  for a.e. $\xi\geq s-t,$ we obtain
		\begin{align}\label{ue3}
		&	\frac{\d}{\d\xi}\E\left[\|\u(\xi,s-t,\u_0)\|^4_{\H}\right]+\E\bigg[4\mu\|\u(\xi,s-t,\u_0)\|^2_{\H}\|\u(\xi,s-t,\u_0)\|^2_{\V}+4\alpha\|\u(\xi,s-t,\u_0)\|^4_{\H}\nonumber\\& \quad\quad+ 4\beta\|\u(\xi,s-t,\u_0)\|^2_{\H}\|\u(\xi,s-t,\u_0)\|^{r+1}_{\widetilde{\L}^{r+1}} \bigg]\nonumber\\&= \E\bigg[4\|\u(\xi,s-t,\u_0)\|^2_{\H}\langle\f(\xi),\u(\xi,s-t,\u_0)\rangle+4\varepsilon^2\|\upsigma^*(\xi,\u(\xi,s-t,\u_0))\u(\xi,s-t,\u_0)\|^2_{\H_0}\nonumber\\&\quad+2\varepsilon^2\|\u(\xi,s-t,\u_0)\|^2_{\H}\|\upsigma(\xi,\u(\xi,s-t,\u_0))\|^2_{\mathcal{L}_2(\H_0,\H)}\bigg],
		\end{align}
		where we used the fact $\int_{s-t}^{\xi}\|\u(\tau)\|^2_{\H}\big(\u(\tau),\upsigma(\tau,\u(\tau))\d\W(\tau)\big)$  is a local martingale. Using H\"oler's and Young's inequalities, we have 
		\begin{align}\label{ue4}
		&	4\|\u(\xi,s-t,\u_0)\|^2_{\H}|\langle\f(\xi),\u(\xi,s-t,\u_0)\rangle|\nonumber\\&\leq4\|\u(\xi,s-t,\u_0)\|^2_{\H}\|\f(\xi)\|_{\V'}\|\u(\xi,s-t,\u_0)\|_{\V}\nonumber\\&\leq \mu\|\u(\xi,s-t,\u_0)\|^2_{\H}\|\u(\xi,s-t,\u_0)\|^2_{\V}+\frac{4}{\mu}\|\f(\xi)\|^2_{\V'}\|\u(\xi,s-t,\u_0)\|^2_{\H}\nonumber\\&\leq \mu\|\u(\xi,s-t,\u_0)\|^2_{\H}\|\u(\xi,s-t,\u_0)\|^2_{\V}+\mu\lambda_1\|\u(\xi,s-t,\u_0)\|^4_{\H}+\frac{4}{\mu^3\lambda_1}\|\f(\xi)\|^4_{\V'}.
		\end{align}
		Let us fix $\varepsilon_0=\sqrt{\frac{\mu\lambda_1}{12K}}$. Using condition (H.2) from Hypothesis \ref{sigmaH}, for $0<\varepsilon\leq\varepsilon_0,$ we get 
		\begin{align}\label{ue5}
		&2\varepsilon^2\|\u(\xi,s-t,\u_0)\|^2_{\H}\|\upsigma(\xi,\u(\xi,s-t,\u_0))\|^2_{\mathcal{L}_2(\H_0,\H)}\nonumber\\&\leq2\varepsilon^2K\bigg[\|\u(\xi,s-t,\u_0)\|^2_{\H}+\|\u(\xi,s-t,\u_0)\|^4_{\H}\bigg]\nonumber\\&\leq2\varepsilon^2K\bigg[\frac{1}{4} + 	2\|\u(\xi,s-t,\u_0)\|^4_{\H}\bigg]\nonumber\\&\leq\frac{\mu\lambda_1}{24} + \frac{\mu\lambda_1}{3}\|\u(\xi,s-t,\u_0)\|^4_{\H}.
		\end{align}
		Using \eqref{ue5},  for all $0<\varepsilon\leq\varepsilon_0,$ we obtain
		\begin{align}\label{ue6}
		&4\varepsilon^2\|\upsigma^*(\xi,\u(\xi,s-t,\u_0))\u(\xi,s-t,\u_0)\|^2_{\H_0}\nonumber\\&\leq 4\varepsilon^2\|\upsigma^*(\xi,\u(\xi,s-t,\u_0))\|^2_{\mathcal{L}_2(\H,\H_0)}\|\u(\xi,s-t,\u_0)\|^2_{\H}\nonumber\\&\leq \frac{\mu\lambda_1}{12} + \frac{2\mu\lambda_1}{3}\|\u(\xi,s-t,\u_0)\|^4_{\H}.
		\end{align}
		Combining  \eqref{ue4}-\eqref{ue6} and substituting it in \eqref{ue3}, we deduce that 
		\begin{align*}
		&\frac{\d}{\d\xi}\E\left[\|\u(\xi,s-t,\u_0)\|^4_{\H}\right]+\E\bigg[3\mu\|\u(\xi,s-t,\u_0)\|^2_{\H}\|\u(\xi,s-t,\u_0)\|^2_{\V}\nonumber\\&\quad+(4\alpha-2\mu\lambda_1)\|\u(\xi,s-t,\u_0)\|^4_{\H} + 4\beta\|\u(\xi,s-t,\u_0)\|^2_{\H}\|\u(\xi,s-t,\u_0)\|^{r+1}_{\widetilde{\L}^{r+1}} \bigg]\nonumber\\&\leq \frac{\mu\lambda_1}{8} +\frac{4}{\mu^3\lambda_1}\|\f(\xi)\|^4_{\V'}. \nonumber
		\end{align*}
		for all $0<\varepsilon\leq\varepsilon_0$ and a.e. $\xi\geq s-t$. Using the Poincar\'e inequality \eqref{poin}, we find 
		\begin{align}\label{ue7}
		\frac{\d}{\d\xi}\E\left[\|\u(\xi,s-t,\u_0)\|^4_{\H}\right]+\mu\lambda_1\E\left[\|\u(\xi,s-t,\u_0)\|^4_{\H}\right]\leq \frac{\mu\lambda_1}{8} +\frac{4}{\mu^3\lambda_1}\|\f(\xi)\|^4_{\V'}.
		\end{align}
		Multiplying \eqref{ue7} by $e^{\mu\lambda_1\xi}$ and then integrating on $(s-t,s)$ with $t\geq0$, we arrive at 
		\begin{align}\label{ue8}
		\quad\E\left[\|\u(s,s-t,\u_0)\|^4_{\H}\right] \leq e^{-\mu\lambda_1 t}\E\left[\|\u_0\|^4_{\H}\right]+ \frac{\mu\lambda_1}{8} +\frac{4}{\mu^3\lambda_1}e^{-\mu\lambda_1 s}\int_{s-t}^{s}e^{\mu\lambda_1\tau}\|\f(\tau)\|^4_{\V'}\d\tau. 
		\end{align}
		Since $\u_0\in \mathcal{B}(s-t)$ and $\mathcal{B}\in \mathfrak{D}$, we get
		\begin{align}\label{ue9}
		e^{-\mu\lambda_1 t} \E\left[\|\u_0\|^4_{\H}\right]&=e^{-\mu\lambda_1 s}e^{\mu\lambda_1(s-t)} \E\left[\|\u_0\|^4_{\H}\right]\leq e^{-\mu\lambda_1 s}e^{\mu\lambda_1(s-t)} \|\mathcal{B}(s-t)\|^4_{\mathrm{L}^4(\Omega,\mathscr{F}_{s-t};\H)}\nonumber\\&\to0 \text{ as } t\to \infty.
		\end{align}
		Therefore, there exists $T=T(s,\mathcal{B})>0$ such that  $e^{-\mu\lambda_1 t} \E\left[\|\u_0\|^4_{\H}\right]\leq 1,$ for all $t\geq T$. By \eqref{ue8}, for $t\geq T$, we obtain
		\begin{align}\label{ue10}
		\E\left[\|\u(s,s-t,\u_0)\|^4_{\H}\right] \leq 1+  \frac{\mu\lambda_1}{8} +\frac{4}{\mu^3\lambda_1}e^{-\mu\lambda_1 s}\int_{-\infty}^{s}e^{\mu\lambda_1\tau}\|\f(\tau)\|^4_{\V'}\d\tau,
		\end{align}which completes the proof.
	\end{proof}
	Next, we prove the main result of this section, that is, the existence of weak $\mathfrak{D}$-pullback mean random attractors for $\Phi_1$.
	\begin{theorem}\label{WPMRA1}
		Suppose that Hypothesis \ref{f_H} and all the conditions of Hypothesis \ref{sigmaH} are satisfied. Then, there exists $\varepsilon_0>0$ such that for every $0<\varepsilon\leq\varepsilon_0$, the mean random dynamical system $\Phi_1$ for the system \eqref{S-CBF} has a unique weak $\mathfrak{D}$-pullback mean random attractor $\mathcal{A}=\{\mathcal{A}(s):s\in\R\}\in \mathfrak{D}$ in $\mathrm{L}^4(\Omega,\mathscr{F};\H)$ over $(\Omega,\mathscr{F},\{\mathscr{F}_t\}_{t\in\R},\mathbb{P}).$  
	\end{theorem}
	\begin{proof}
		For a given $s\in\R$, let us define
		\begin{align*}
		\mathcal{K}(s) :=\left\{\u\in\mathrm{L}^4(\Omega,\mathscr{F}_s;\H):\E\left[\|\u\|^4_{\H}\right]\leq R(s)\right\},
		\end{align*}
		where $$R(s)=1+ \frac{\mu\lambda_1}{8} +\frac{4}{\mu^3\lambda_1}e^{-\mu\lambda_1 s}\int_{-\infty}^{s}e^{\mu\lambda_1\tau}\|\f(\tau)\|^4_{\V'}\d\tau.$$
		Since $\mathcal{K}(s)$ is a bounded closed convex subset of the reflexive Banach space $\mathrm{L}^4(\Omega,\mathscr{F}_s;\H)$, we infer that $\mathcal{K}(s)$ is weakly compact in $\mathrm{L}^4(\Omega,\mathscr{F}_s;\H)$. By Hypothesis \ref{f_H}, we also obtain 
		\begin{align*}
		&	\lim_{s\to-\infty}e^{\mu\lambda_1 s}\|\mathcal{K}(s)\|^4_{\mathrm{L}^4(\Omega,\mathscr{F}_s;\H)}\\&=\lim_{s\to-\infty}e^{\mu\lambda_1 s}R(s)=\lim_{s\to-\infty}\bigg[\bigg(1+\frac{\mu\lambda_1}{8}\bigg)e^{\mu\lambda_1 s} +\frac{4}{\mu^3\lambda_1}\int_{-\infty}^{s}e^{\mu\lambda_1\tau}\|\f(\tau)\|^4_{\V'}\d\tau\bigg]=0,
		\end{align*}
		which implies that $\mathcal{K}=\{\mathcal{K}(s):s\in\R\}\in\mathfrak{D}$. Hence, by Lemma \ref{absorb}, we infer that $\mathcal{K}$ is a weakly compact $\mathfrak{D}$-pullback absorbing set for $\Phi_1$. Making use of Theorem \ref{Main-T}, we immediately conclude that there exists a unique weak $\mathfrak{D}$-pullback mean random attractor $\mathcal{A}\in\mathfrak{D}$ of $\Phi_1$.
	\end{proof}

	\section{Weak Pullback Mean Random Attractors for $r>3$ ($n=2,3$) and for $r=3$ ($n=3$ and $2\beta\mu\geq1$)}\label{sec5}\setcounter{equation}{0}
	This section is devoted for finding the mean random dynamical system for the nonautonomous SCBF equations \eqref{1} over a filtered probability space for $r>3$ ($n=2,3$) and for $r=3$ ($n=3$ and $2\beta\mu\geq1$), and obtaining  the existence and uniqueness of weak pullback mean random attractors. Next, we discuss about the solvability result of the system \eqref{S-CBF}.
	\begin{definition}
		For $r>3$ ($n=2,3$) and for $r=3$ ($n=3$ and $2\beta\mu\geq1$), let $s\in\R$ and $\u_0\in \mathrm{L}^2(\Omega, \mathscr{F}_s;\H).$ Then, an $\H$-valued $\{\mathscr{F}_t\}_{t\in\R}$-adapted stochastic process $\{\u(t)\}_{t\in[s,\infty)}$ is called a strong solution of the system  \eqref{S-CBF} on $[s,\infty)$ with initial data $\u_0$ if $$\u\in\mathrm{C}([s,\infty);\H)\cap\mathrm{L}^2_{\mathrm{loc}}((s,\infty);\V)\cap\mathrm{L}^{r+1}_{\mathrm{loc}}((s,\infty);\widetilde{\L}^{r+1}), \  \mathbb{P}\text{-a.s.},$$ and satisfies, for every $t>s$ and $\v\in\V\cap\widetilde{\L}^{r+1}$,
		\begin{align*}
		(\u(t),\v)
		&	+ \int_{s}^{t} \langle\mu\A\u(\tau)+ \B\u(\tau)+\alpha\u(\tau)+\beta\mathcal{C}(\u(\tau)),\v\rangle\d \tau \\&= (\u_0,\v)+ \int_{s}^{t}\langle\f(\tau), \v\rangle\d\tau +\varepsilon\int_{s}^{t} (\v, \upsigma(\tau,\u(\tau))\d\W(\tau)),
		\end{align*}$\mathbb{P}$-a.s., and  $\u(\cdot)$ satisfies the following It\^o formula (energy equality): 
		\begin{align}\label{41}
		&\|\u(t)\|_{\H}^2+2\mu\int_s^t\|\u(\tau)\|_{\V}^2\d\tau+2\alpha\int_s^t\|\u(\tau)\|_{\H}^2\d\tau+2\beta\int_s^t\|\u(\tau)\|_{\wi\L^{r+1}}^{r+1}\d\tau\nonumber\\&=\|\u_0\|_{\H}^2+\int_s^t\langle\f(\tau),\u(\tau)\rangle\d\tau+\frac{\e^2}{2}\int_s^{t}\|\upsigma(\tau,\u(\tau))\|_{\mathcal{L}_2(\H_0,\H)}^2+\e\int_s^t(\u(\tau),\upsigma(\tau,\u)\d\W(\tau)),
		\end{align}
		for all $t>s$,	$\mathbb{P}$-a.s. A strong solution $\u(\cdot)$ to the system \eqref{S-CBF} is called a
		\emph{pathwise  unique strong solution} if
		$\widetilde{\u}(\cdot)$ is an another strong
		solution, then $$\mathbb{P}\big\{\omega\in\Omega:\u(t)=\widetilde{\u}(t),\ \text{ for all }\ t\in[0,T]\big\}=1.$$ 
	\end{definition}
	Under the Hypothesis \ref{sigmaH} on $\upsigma(\cdot,\cdot)$, let us now provide the following result on the existence and uniqueness of strong  solutions of system \eqref{S-CBF}, which is proved in \cite{MTM} (Theorem 3.7 for $r>3$ and Theorem 3.9 for $r=3$ with $2\beta\mu\geq1$).
	\begin{proposition}[\cite{MTM}]\label{MRDS_1}
		For $r>3$ ($n=2,3$) and for $r=3$ ($n=3$ and $2\beta\mu\geq1$), let all the conditions of Hypothesis \ref{sigmaH} be satisfied. Then, there exists $\widetilde{\varepsilon}_0>0$ such that for every $\varepsilon\in(0,\widetilde{\varepsilon}_0), s\in\R$, $\u_0\in \mathrm{L}^2(\Omega,\mathscr{F}_s;\H)$ and $\f\in \mathrm{L}^2_{\mathrm{loc}}(\R,\V')$, the system \eqref{S-CBF} has a unique pathwise strong solution $$\u\in\mathrm{L}^2(\Omega; \mathrm{C}([s,s+T];\H)\cap\mathrm{L}^2((s,s+T);\V))\cap\mathrm{L}^{r+1}(\Omega;\mathrm{L}^{r+1}((s,s+T);\widetilde{\L}^{r+1})),$$ for every $T>0.$ Moreover, 
		\begin{align}\label{MRDS1_1}
		\mathbb{E}\bigg[\sup_{\tau\in[s,s+T]}\|\u(\tau)\|^2_{\H}&+ 2\mu\int_{s}^{s+T}\|\u(\tau)\|^2_{\V}\d\tau+2\alpha\int_{s}^{s+T}\|\u(\tau)\|^2_{\H}\d\tau\nonumber\\&+2\beta\int_{s}^{s+T}\|\u(\tau)\|^{r+1}_{\widetilde{\L}^{r+1}}\d\tau\bigg]\leq \widetilde{\M}(1+\mathbb{E}(\|\u_0\|^2_{\H})),
		\end{align} where $\widetilde{\M}>0$ is a constant independent of $\u_0$.
	\end{proposition}
	It should be noted that $\u\in\mathrm{C}([s,\infty),\H),$ $\mathbb{P}$-a.s. Therefore, using \eqref{MRDS1_1} and the Lebesgue dominated convergence theorem, we deduce that $\u\in\mathrm{C}([s,\infty),\mathrm{L}^2(\Omega,\mathscr{F};\H)),$ which  help us to define a mean random dynamical system for \eqref{S-CBF}. Given $t\in\R^+$ and $s\in\R,$ let $\Phi_2(t,s)$ be a mapping from $\mathrm{L}^2(\Omega,\mathscr{F}_s;\H)$ to $\mathrm{L}^2(\Omega,\mathscr{F}_{s+t};\H)$ given by $$\Phi_2(t,s)(\u_0)=\u(t+s,s,\u_0),$$ where $\u_0\in \mathrm{L}^2(\Omega,\mathscr{F}_s;\H)$, and $\u(\cdot)$ is the unique pathwise strong solution to the system \eqref{S-CBF}. Since the strong solution of system \eqref{S-CBF} is unique, we deduce  that for every $t_1,t_2\geq0$ and $s\in\R$, $$\Phi_2(t_1+t_2,s)=\Phi_2(t_1,t_2+s)\circ\Phi_2(t_2,s).$$ Consequently, $\Phi_2$ is a mean random dynamical system on $\mathrm{L}^2(\Omega,\mathscr{F};\H)$ over $(\Omega, \mathscr{F},\{\mathscr{F}_t\}_{t\in\R},\mathbb{P})$ in the sense of Definition \ref{def_MRDS}.
	
	Let $\widetilde{\mathcal{B}}=\{\widetilde{\mathcal{B}}(s)\subseteq\mathrm{L}^2(\Omega,\mathscr{F}_s;\H):s\in\R\}$ be a family of nonempty bounded sets such that \begin{align}\label{classB_1}
	\lim_{s\to-\infty}e^{\mu\lambda_1 s} \|\widetilde{\mathcal{B}}(s)\|^2_{\mathrm{L}^2(\Omega,\mathscr{F}_s;\H)}=0,
	\end{align} where $$\|\widetilde{\mathcal{B}}(s)\|_{\mathrm{L}^2(\Omega,\mathscr{F}_s;\H)}=\sup_{\u\in\widetilde{\mathcal{B}}(s)}\|\u\|_{\mathrm{L}^2(\Omega,\mathscr{F}_s;\H)}.$$
	Let $\widetilde{\mathfrak{D}}$ be the collection of all families of nonempty bounded sets with the property \eqref{classB_1}:
	\begin{align}
	\widetilde{\mathfrak{D}}=\{\widetilde{\mathcal{B}}=\{\widetilde{\mathcal{B}}\subseteq\mathrm{L}^2(\Omega,\mathscr{F}_s;\H):\widetilde{\mathcal{B}}\neq\emptyset \text{ bounded}, s\in\R\}:\widetilde{\mathcal{B}} \text{ satisfies } \eqref{classB_1}\}.
	\end{align}
	We need the following assumption on $\f(\cdot)$ to prove the existence and uniqueness of weak $\widetilde{\mathfrak{D}}$-pullback mean random attractors of $\Phi_2$.
	\begin{hypothesis}\label{f_H_1}
		The deterministic forcing term $\f$ satisfies
		\begin{align}\label{f_H1_1}
		\int_{-\infty}^{s} e^{\mu\lambda_1\tau}\|\f(\tau)\|^2_{\V'} \d\tau<\infty,\  \text{ for all } \ s\in\R.
		\end{align}
		From \eqref{f_H1}, it is immediate that 
		\begin{align}\label{f_H2_1}
		\lim_{s\to-\infty} 	\int_{-\infty}^{s} e^{\mu\lambda_1\tau}\|\f(\tau)\|^2_{\V'} \d\tau=0.
		\end{align}
	\end{hypothesis}

	\subsection{Weak $\widetilde{\mathfrak{D}}$-pullback mean random attractors}
	In this subsection, we prove the existence and uniqueness of weak $\widetilde{\mathfrak{D}}$-pullback mean random attractors for the system \eqref{S-CBF} in $\mathrm{L}^2(\Omega,\mathscr{F};\H)$ over $(\Omega,\mathscr{F},\{\mathscr{F}_t\}_{t\in\R},\mathbb{P}).$
	\begin{lemma}\label{absorb_1}
		For $r>3$ ($n=2,3$) and for $r=3$ ($n=3$ and $2\beta\mu\geq1$), assume that \eqref{f_H1_1} and all the conditions of Hypothesis \ref{sigmaH} are satisfied. Then, there exists $\widetilde{\varepsilon}_0>0$ such that for every $0<\varepsilon\leq\widetilde{\varepsilon}_0$ and for every $s\in\R$ and $\widetilde{\mathcal{B}}=\{\widetilde{\mathcal{B}}(t)\}_{t\in\R}\in \widetilde{\mathfrak{D}},$ there exists $T=T(s,\widetilde{\mathcal{B}})>0$ such that for all $t\geq T,$ the strong solution $\u(\cdot)$ of system \eqref{S-CBF} satisfies  
		\begin{align}\label{ue_1}
		\mathbb{E}\left[\|\u(s,s-t,\u_0)\|_{\H}^2\right]\leq 1 +\frac{\mu\lambda_1}{2} +\frac{2}{\mu}e^{-\mu\lambda_1 s}\int_{-\infty}^{s}e^{\mu\lambda_1\tau}\|\f(\tau)\|^2_{\V'}\d\tau,
		\end{align} where $\u_0\in\widetilde{\mathcal{B}}(s-t)$.
	\end{lemma}
	\begin{proof}
		Applying the infinite dimensional It\^o formula to the process $\|\u(\cdot)\|^2_{\H}$ (see \cite{MTM} and \eqref{41} also), we obtain
		\begin{align}\label{ue1_1}
		\d\|\u(\xi)\|_{\H}^2=&\big(-2\mu\|\u(\xi)\|^2_{\V}-2\alpha\|\u(\xi)\|^2_{\H}-2\beta\|\u(\xi)\|^{r+1}_{\widetilde{\L}^{r+1}}+2\langle\f(\xi),\u(\xi)\rangle\nonumber\\&\quad+\varepsilon^2\|\upsigma(\xi,\u(\xi))\|^2_{\mathcal{L}_2(\H_0,\H)}\big)\d\xi +2\varepsilon\big(\u(\xi),\upsigma(\xi,\u(\xi))\d\W(\xi)\big).
		\end{align}
		Taking the expectation of \eqref{ue1_1} and using the fact that $\int_{s-t}^{\xi}\big(\u(\tau),\upsigma(\tau,\u(\tau))\d\W(\tau)\big)$ is a local martingale, we obtain
		\begin{align}\label{ue2_1}
		&\frac{\d}{\d\xi}\E\left[\|\u(\xi,s-t,\u_0)\|^2_{\H}\right]+\E\bigg[2\mu\|\u(\xi,s-t,\u_0)\|^2_{\V}+2\alpha\|\u(\xi,s-t,\u_0)\|^2_{\H}\nonumber\\&\quad+ 2\beta\|\u(\xi,s-t,\u_0)\|^{r+1}_{\widetilde{\L}^{r+1}} \bigg]\nonumber\\&= \E\bigg[2\langle\f(\xi),\u(\xi,s-t,\u_0)\rangle+\varepsilon^2\|\upsigma(\xi,\u(\xi,s-t,\u_0))\|^2_{\mathcal{L}_2(\H_0,\H)}\bigg],
		\end{align}
		for a.e. $\xi\geq s-t$.	Using H\"oler's and Young's inequalities, we obtain
		\begin{align}\label{ue3_1}
		2|\langle\f(\xi),\u(\xi,s-t,\u_0)\rangle|&\leq 2\|\f(\xi)\|_{\V'}\|\u(\xi,s-t,\u_0)\|_{\V}\leq \frac{\mu}{2}\|\u(\xi,s-t,\u_0)\|^2_{\V}+\frac{2}{\mu}\|\f(\xi)\|^2_{\V'}.
		\end{align}
		Let us take $\widetilde{\varepsilon}_0=\sqrt{\frac{\mu\lambda_1}{2K}}$. Using  Hypothesis \ref{sigmaH} (H.2),  for $0<\varepsilon\leq\widetilde{\varepsilon}_0,$ we obtain
		\begin{align}\label{ue4_1}
		\varepsilon^2\|\upsigma(\xi,\u(\xi,s-t,\u_0))\|^2_{\mathcal{L}_2(\H_0,\H)}&\leq\varepsilon^2K\bigg[1+\|\u(\xi,s-t,\u_0)\|^2_{\H}\bigg]\nonumber\\&\leq\frac{\mu\lambda_1}{2} + \frac{\mu\lambda_1}{2}\|\u(\xi,s-t,\u_0)\|^2_{\H}.
		\end{align}
		Combining \eqref{ue3_1}-\eqref{ue4_1} and then substituting it in \eqref{ue2_1}, we obtain
		\begin{align*}
		&\quad	\frac{\d}{\d\xi}\E\left[\|\u(\xi,s-t,\u_0)\|^2_{\H}\right]+\E\bigg[\frac{3\mu}{2}\|\u(\xi,s-t,\u_0)\|^2_{\V}+\left(2\alpha-\frac{\mu\lambda_1}{2}\right)\|\u(\xi,s-t,\u_0)\|^2_{\H}\nonumber\\& \quad\quad+ 2\beta\|\u(\xi,s-t,\u_0)\|^{r+1}_{\widetilde{\L}^{r+1}} \bigg]\leq \frac{\mu\lambda_1}{2} +\frac{2}{\mu}\|\f(\xi)\|^2_{\V'},
		\end{align*}
		for all $0<\varepsilon\leq\widetilde{\varepsilon}_0$ and for a.e. $\xi\geq s-t$. Making use of the Poincar\'e inequality \eqref{poin}, we obtain
		\begin{align}\label{ue5_1}
		\frac{\d}{\d\xi}\E\left[\|\u(\xi,s-t,\u_0)\|^2_{\H}\right]+\mu\lambda_1\E\left[\|\u(\xi,s-t,\u_0)\|^2_{\H}\right]\leq \frac{\mu\lambda_1}{2} +\frac{2}{\mu}\|\f(\xi)\|^2_{\V'}.
		\end{align}
		Multiplying \eqref{ue5_1} by $e^{\mu\lambda_1\xi}$ and then integrating on $(s-t,s)$ with $t\geq0$, we obtain
		\begin{align}\label{ue6_1}
		\E\left[\|\u(s,s-t,\u_0)\|^2_{\H}\right] \leq e^{-\mu\lambda_1 t}\E\left[\|\u_0\|^2_{\H}\right] + \frac{\mu\lambda_1}{2} +\frac{2}{\mu}e^{-\mu\lambda_1 s}\int_{s-t}^{s}e^{\mu\lambda_1\tau}\|\f(\tau)\|^2_{\V'}\d\tau. 
		\end{align}
		Since $\u_0\in \widetilde{\mathcal{B}}(s-t)$ and $\widetilde{\mathcal{B}}\in \widetilde{\mathfrak{D}}$, we get
		\begin{align}\label{ue7_1}
		e^{-\mu\lambda_1 t} \E\left[\|\u_0\|^2_{\H}\right]&=e^{-\mu\lambda_1 s}e^{\mu\lambda_1(s-t)} \E\left[\|\u_0\|^2_{\H}\right]\leq e^{-\mu\lambda_1 s}e^{\mu\lambda_1(s-t)} \|\widetilde{\mathcal{B}}(s-t)\|^2_{\mathrm{L}^2(\Omega,\mathscr{F}_{s-t};\H)}\nonumber\\&\to 0 \ \text{ as } \ t\to \infty.
		\end{align}
		Therefore, there exists $T=T(s,\widetilde{\mathcal{B}})>0$ such that  $e^{-\mu\lambda_1 t} \E\left[\|\u_0\|^2_{\H}\right]\leq 1,$ for all $t\geq T$. By \eqref{ue6_1}, we finally obtain, 
		\begin{align}\label{ue8_1}
		\E\left[\|\u(s,s-t,\u_0)\|^2_{\H}\right] \leq 1+  \frac{\mu\lambda_1}{2} +\frac{2}{\mu}e^{-\mu\lambda_1 s}\int_{-\infty}^{s}e^{\mu\lambda_1\tau}\|\f(\tau)\|^2_{\V'}\d\tau,
		\end{align}for $t\geq T$, which completes the proof.
	\end{proof}
	Let us now prove the main result of this section, that is, the existence of weak $\widetilde{\mathfrak{D}}$-pullback mean random attractors for $\Phi_2$.
	\begin{theorem}\label{WPMRA2}
		For $r>3$ ($n=2,3$) and for $r=3$ ($n=3$ and $2\beta\mu\geq1$), suppose that Hypothesis \ref{f_H_1} and all the conditions of Hypothesis \ref{sigmaH} are satisfied. Then, there exists $\widetilde{\varepsilon}_0>0$ such that for every $0<\varepsilon\leq\widetilde{\varepsilon}_0$, the mean random dynamical system $\Phi_2$ for the system \eqref{S-CBF} has a unique weak $\widetilde{\mathfrak{D}}$-pullback mean random attractor $\widetilde{\mathcal{A}}=\{\widetilde{\mathcal{A}}(s):s\in\R\}\in \widetilde{\mathfrak{D}}$ in $\mathrm{L}^2(\Omega,\mathscr{F};\H)$ over $(\Omega,\mathscr{F},\{\mathscr{F}_t\}_{t\in\R},\mathbb{P}).$  
	\end{theorem}
	\begin{proof}
		For a given $s\in\R$, let us define
		\begin{align*}
		\widetilde{\mathcal{K}}(s) :=\{\u\in\mathrm{L}^2(\Omega,\mathscr{F}_s;\H):\E\left[\|\u\|^2_{\H}\right]\leq \widetilde{R}(s)\},
		\end{align*}
		where $$\widetilde{R}(s)=1+ \frac{\mu\lambda_1}{2} +\frac{2}{\mu}e^{-\mu\lambda_1 s}\int_{-\infty}^{s}e^{\mu\lambda_1\tau}\|\f(\tau)\|^2_{\V'}\d\tau.$$
		Since $\widetilde{\mathcal{K}}(s)$ is a bounded closed convex subset of the reflexive Banach space $\mathrm{L}^2(\Omega,\mathscr{F}_s;\H)$, we infer that $\widetilde{\mathcal{K}}(s)$ is weakly compact in $\mathrm{L}^2(\Omega,\mathscr{F}_s;\H)$. By Hypothesis \ref{f_H_1}, we also get
		\begin{align*}
		\lim_{s\to-\infty}e^{\mu\lambda_1 s}\|\widetilde{\mathcal{K}}(s)\|^2_{\mathrm{L}^2(\Omega,\mathscr{F}_s;\H)}&=\lim_{s\to-\infty}e^{\mu\lambda_1 s}\widetilde{R}(s)\\&=\lim_{s\to-\infty}\bigg[\bigg(1+\frac{\mu\lambda_1}{2}\bigg)e^{\mu\lambda_1 s} +\frac{2}{\mu}\int_{-\infty}^{s}e^{\mu\lambda_1\tau}\|\f(\tau)\|^2_{\V'}\d\tau\bigg]\\&=0,
		\end{align*}
		which implies that $\widetilde{\mathcal{K}}=\{\widetilde{\mathcal{K}}(s):s\in\R\}\in\widetilde{\mathfrak{D}}$. Hence, by Lemma \ref{absorb_1}, we infer that $\widetilde{\mathcal{K}}$ is a weakly compact $\widetilde{\mathfrak{D}}$-pullback absorbing set for $\Phi_2$. Now, by Theorem \ref{Main-T}, one can easily come to a conclusion that there exists a unique weak $\widetilde{\mathfrak{D}}$-pullback mean random attractor $\widetilde{\mathcal{A}}\in\widetilde{\mathfrak{D}}$ of $\Phi_2$.
	\end{proof}

	\section{Weak Pullback Mean Random Attractors for Locally Monotone Stochastic Partial Differential Equations}\label{sec6}\setcounter{equation}{0}
	In this section, we prove the existence of weak pullback mean random attractors for locally monotone stochastic partial differential equations (SPDEs) discussed in \cite{IDC,LR,LR1}, etc. Let us first provide the functional framework of the locally  monotone SPDEs. 
	
	Let $(\mathbb{X}, (\cdot,\cdot))$ be a real separable Hilbert space, identified with its own dual space $\mathbb{X}'$. Let $\mathbb{Y}$ be a real reflexive Banach space continuously and densely embedded into $\X$. In particular, there is a constant $\lambda>0$ such that 
	\begin{align}\label{poinLM}
	\lambda\|\y\|_{\X}^2\leq\|\y\|_{\mathbb{Y}}^2, \ \text{ for all }\ \y\in\mathbb{Y},
	\end{align}
	so that we have the Gelfand triple $\mathbb{Y}\subseteq\X\equiv\X'\subseteq\mathbb{Y}'.$ Let $\langle\cdot,\cdot\rangle$ denote the duality pairing between $\mathbb{Y}$ and its dual space $\mathbb{Y}'$. We consider the following stochastic partial differential equation
	\begin{align}\label{LM_EQ}
	\d\mathrm{U}(t) = (\F(\mathrm{U}(t))+\g(t)) \d t +\varepsilon\upsigma(t,\mathrm{U}(t))\d\mathrm{W}(t), \ \U(0)=\U_0
	\end{align}
	where $\varepsilon>0$ is a constant, $\upsigma$ is a nonlinear diffusion term, and $\W(\cdot)$ is a two-sided Wiener process of trace class defined on some complete filtered probability space (see subsection \ref{WP}).
	
	Let $\X_0=Q^{1/2}\X$ and $\mathcal{L}_2(\X_0,\X)$ be the space of Hilbert-Schmidt operators from $\X_0$ to $\X$ with norm $\|\cdot\|_{\mathcal{L}_2(\X_0,\X)}$ given by 
	$$\|\Psi\|^2_{\mathcal{L}_2(\X_0,\X)}=\Tr(\Psi Q\Psi^*),\ \text{ for all } \ \Psi\in \mathcal{L}_2(\X_0,\X),$$ where $\Psi^*$ is the adjoint operator of $\Psi$. For more details, the interested readers are referred to see \cite{DZ1}.

	\begin{hypothesis}\label{LM_F_H}
		The coefficient $\F(\cdot)$ satisfies the following conditions hold for all $\y,\y_1,\y_2\in \mathbb{Y}$:
		\begin{itemize}
			\item [(F.1)] \emph{(Hemicontinuity).} The map $s\mapsto\langle\F(\y_1+s\y_2),\y\rangle$ is continuous on $\R$.
			\item[(F.2)] \emph{(Local monotonicity).} There exist a measurable and locally bounded (in $\mathbb{Y}$) function $\uprho:\mathbb{Y}\to[0,+\infty)$  and a constant $\theta_4\geq 0$ such that 
			\begin{align*}
			2\langle\F(\y_1)-\F(\y_2),\y_1-\y_2\rangle\leq(\theta_4+\uprho(\y_2))\|\y_1-\y_2\|^2_{\X}.
			\end{align*}
			\item[(F.3)] \emph{(Coercivity).} There exist constants $\theta_1>0$, $\theta_2,\theta_3\geq 0$ such that 
			\begin{align*}
			2\langle\F(\y),\y\rangle\leq-\theta_1\|\y\|^{\gamma}_{\mathbb{Y}}+\theta_2\|\y\|^2_{\X} +\theta_3.
			\end{align*}
			\item [(F.4)] \emph{(Growth).} There exist constants $\gamma\geq 2$, $\delta,\theta_5\geq 0$ such that \begin{align*}
			\|\F(\y)\|_{\mathbb{Y}'}^{\frac{\gamma}{\gamma-1}}\leq \theta_5(1+\|\y\|^{\gamma}_{\mathbb{Y}})(1+\|\y\|^{\delta}_{\X}).
			\end{align*}
		\end{itemize}
	\end{hypothesis}
	\begin{hypothesis}\label{LMsigmaH}
		The noise coefficient $\upsigma(\cdot,\cdot)$ satisfies the following:
		\begin{itemize}
			\item [(M.1)] The function $\upsigma\in\mathrm{C}(\R\times\mathbb{Y};\mathcal{L}_2(\X_0,\X)).$
			\item[(M.2)] \emph{(Growth condition)} There exists a positive constant K such that for all $\x\in\X$ and $t\in\R$,$$\|\upsigma(t,\x)\|^2_{\mathcal{L}_2(\X_0,\X)}\leq K(1+\|\x\|_{\X}^2).$$
			\item[(M.3)] \emph{(Lipschitz condition)} There exists a positive constant L such that for all $\x_1,\x_2\in\X$ and $t\in\R$,$$\|\upsigma(t,\x_1)-\upsigma(t,\x_2)\|^2_{\mathcal{L}_2(\X_0,\X)}\leq L\|\x_1-\x_2\|_{\X}^2.$$
		\end{itemize}
	\end{hypothesis}
	\begin{hypothesis}\label{LMf_H}
		The deterministic forcing term $\g$ satisfies:
		\begin{itemize}
			\item [1.] For $\gamma>2$, we assume that 
			\begin{align}\label{LMf_H1}
			\int_{-\infty}^{s}e^{-\theta_1\lambda^{\frac{\gamma}{2}} \tau}\|\g(\tau)\|^{\frac{p\gamma}{2(\gamma-1)}}_{\mathbb{Y}'}\d \tau<\infty,\  \text{ for all } \ s\in\R.
			\end{align}
			From \eqref{LMf_H1}, we also have 
			\begin{align}\label{LMf_H2}
			\lim_{s\to-\infty} 	\int_{-\infty}^{s}e^{-\theta_1\lambda^{\frac{\gamma}{2}} \tau}\|\g(\tau)\|^{\frac{p\gamma}{2(\gamma-1)}}_{\mathbb{Y}'}\d \tau=0.
			\end{align}
			\item[2.] For $\gamma=2$ and $\widetilde{\theta}:=\theta_1-\frac{\theta_2+\theta_3}{\lambda}>0$, we suppose that 
			\begin{align}\label{LMf_H1_1}
			\int_{-\infty}^{s}e^{-\widetilde{\theta}\lambda \tau}\|\g(\tau)\|^{p}_{\mathbb{Y}'}\d \tau<\infty,\  \text{ for all }\  s\in\R.
			\end{align}
			From \eqref{LMf_H1_1}, we also obtain 
			\begin{align}\label{LMf_H2_1}
			\lim_{s\to-\infty} 	\int_{-\infty}^{s}e^{-\widetilde{\theta}\lambda \tau}\|\g(\tau)\|^{p}_{\mathbb{Y}'}\d \tau=0.
			\end{align}
			
		\end{itemize}
	\end{hypothesis}
	
	Under the Hypothesis \ref{LM_F_H}-\ref{LMf_H}, we provide the following result on the existence and uniqueness of strong solutions of system \eqref{LM_EQ}, which is proved in \cite{LR}.
	
	\begin{definition}
		Let $s\in\R$ and $\U_0\in \mathrm{L}^p(\Omega, \mathscr{F}_s;\X).$ Then, an $\H$-valued $\{\mathscr{F}_t\}_{t\in\R}$-adapted stochastic process $\{\U(t)\}_{t\in[s,\infty)}$ is called a \emph{strong solution} of \eqref{LM_EQ} on $[s,\infty)$ with initial data $\U_0$ if $\U\in\mathrm{C}([s,\infty);\H)\cap\mathrm{L}^{\gamma}_{\mathrm{loc}}((s,\infty);\V), \  \mathbb{P}$-a.s., and satisfies, for every $t>s$ and $\y\in\Y'$,
		\begin{align*}
		(\U(t),\y)=(\U_0,\y)+\int_s^t\langle\F(\U(\tau)),\y\rangle\d \tau+\e\int_s^t(\y,\upsigma(s,\U(s))\d\W(s)),
		\end{align*}$\mathbb{P}$-a.s. A strong solution $\U(\cdot)$ to the system \eqref{LM_EQ} is called a
		\emph{pathwise  unique strong solution} if
		$\widetilde{\U}(\cdot)$ is an another strong
		solution, then $$\mathbb{P}\big\{\omega\in\Omega:\U(t)=\widetilde{\U}(t),\ \text{ for all }\ t\in[0,T]\big\}=1.$$ 
	\end{definition} 
	
	\begin{proposition}[Theorem 1.1, \cite{LR}]\label{MRDS_LM}
		Let all the conditions of Hypothesis \ref{LM_F_H}-\ref{LMf_H} be satisfied, and there exists a positive constant $C$ such that $$\uprho(\y)\leq C(1+\|\y\|^{\gamma}_{\mathbb{Y}})(1+\|\y\|^{\delta}_{\X}), \ \text{ for all }\ \y\in\mathbb{Y}.$$ Then, there exists $\varepsilon_0>0$ such that for every $\varepsilon\in(0,\varepsilon_0), s\in\R$ and $\mathrm{U}_0\in \mathrm{L}^p(\Omega,\mathscr{F}_s;\X)$, the system \eqref{LM_EQ} has a unique strong solution $\{\mathrm{U}(\tau)\}_{\tau\in[s,s+T]}$ and satisfies 
		\begin{align}\label{MRDS_LM1}
		\mathbb{E}\bigg[\sup_{\tau\in[s,s+T]}\|\mathrm{U}(\tau)\|^p_{\X}+ \int_{s}^{s+T}\|\mathrm{U}(\tau)\|^{\gamma}_{\mathbb{Y}}\d\tau\bigg]<\infty,
		\end{align}
		for some $p\geq\delta+2$.
	\end{proposition}
	Given $t\in\R^+$ and $s\in\R,$ let $\Phi(t,s)$ be a mapping from $\mathrm{L}^p(\Omega,\mathscr{F}_s;\X)$ to $\mathrm{L}^p(\Omega,\mathscr{F}_{s+t};\X)$ given by $$\Phi(t,s)(\mathrm{U}_0)=\mathrm{U}(t+s,s,\mathrm{U}_0),$$ where $\mathrm{U}_0\in \mathrm{L}^p(\Omega,\mathscr{F}_s;\X)$, and $\mathrm{U}(\cdot)$ is the unique strong solution to the system \eqref{S-CBF}. Since the strong solution of the system \eqref{LM_EQ} is unique, for every $t_1,t_2\geq0$ and $s\in\R$, we get  $$\Phi(t_1+t_2,s)=\Phi(t_1,t_2+s)\circ\Phi(t_2,s).$$ Consequently, $\Phi$ is a mean random dynamical system on $\mathrm{L}^p(\Omega,\mathscr{F};\X)$ over $(\Omega, \mathscr{F},\{\mathscr{F}_t\}_{t\in\R},\mathbb{P})$ in the sense of Definition \ref{def_MRDS}.
	
	For $\gamma>2$,  let $\mathcal{B}_{\gamma}=\{\mathcal{B}_{\gamma}(s)\subseteq\mathrm{L}^p(\Omega,\mathscr{F}_s;\X):s\in\R\}$ be a family of nonempty bounded sets such that \begin{align}\label{classM}
	\lim_{s\to-\infty}e^{\theta_1\lambda^{\frac{\gamma}{2}} s} \|\mathcal{B}_{\gamma}(s)\|^p_{\mathrm{L}^p(\Omega,\mathscr{F}_s;\X)}=0,
	\end{align} where $$\|\mathcal{B}_{\gamma}(s)\|_{\mathrm{L}^p(\Omega,\mathscr{F}_s;\X)}=\sup_{\mathrm{U}\in\mathcal{B}_{\gamma}(s)}\|\mathrm{U}\|_{\mathrm{L}^p(\Omega,\mathscr{F}_s;\X)}.$$
	Let $\mathfrak{D}_{\gamma}$ be the collection of all families of nonempty bounded sets with the property \eqref{classM}:
	\begin{align}
	\mathfrak{D}_{\gamma}=\{\mathcal{B}_{\gamma}=\{\mathcal{B}_{\gamma}\subseteq\mathrm{L}^p(\Omega,\mathscr{F}_s;\X):\mathcal{B}_{\gamma}\neq\emptyset \text{ bounded}, s\in\R\}:\mathcal{B}_{\gamma} \text{ satisfies } \eqref{classM}\}.
	\end{align}
	Also, for $\gamma=2$, let $\mathcal{B}_2=\{\mathcal{B}_2(s)\subseteq\mathrm{L}^p(\Omega,\mathscr{F}_s;\X):s\in\R\}$ be a family of nonempty bounded sets such that \begin{align}\label{classM_1}
	\lim_{s\to-\infty}e^{\widetilde{\theta}\lambda s} \|\mathcal{B}_2(s)\|^p_{\mathrm{L}^p(\Omega,\mathscr{F}_s;\X)}=0,
	\end{align} where $$\widetilde{\theta}=\theta_1-\frac{\theta_2+\theta_3}{\lambda}\ \ \text{ and }\ \ \|\mathcal{B}_2(s)\|_{\mathrm{L}^p(\Omega,\mathscr{F}_s;\X)}=\sup_{\mathrm{U}\in\mathcal{B}_2(s)}\|\mathrm{U}\|_{\mathrm{L}^p(\Omega,\mathscr{F}_s;\X)}.$$
	Let $\mathfrak{D}_2$ be the collection of all families of nonempty bounded sets with the property \eqref{classM_1}:
	\begin{align}
	\mathfrak{D}_2=\{\mathcal{B}_2=\{\mathcal{B}_2\subseteq\mathrm{L}^p(\Omega,\mathscr{F}_s;\X):\mathcal{B}_2\neq\emptyset \text{ bounded}, s\in\R\}:\mathcal{B}_2 \text{ satisfies } \eqref{classM_1}\}.
	\end{align}
	\subsection{Weak $\mathfrak{D}$-pullback mean random attractors}
	In this subsection, we show the existence and uniqueness of weak $\mathfrak{D}_{\gamma}$-pullback mean random attractors for the system \eqref{LM_EQ} in $\mathrm{L}^p(\Omega,\mathscr{F};\X)$ over the filtered probability space  $(\Omega,\mathscr{F},\{\mathscr{F}_t\}_{t\in\R},\mathbb{P}).$
	\begin{lemma}\label{absorb_LM}
		Let all the conditions of Hypothesis \ref{LM_F_H}-\ref{LMf_H} are satisfied with $\gamma>2$. Then, there exists $\varepsilon_0>0$ such that for every $0<\varepsilon\leq\varepsilon_0$ and for every $s\in\R$ and $\mathcal{B}_{\gamma}=\{\mathcal{B}_{\gamma}(t)\}_{t\in\R}\in \mathfrak{D}_{\gamma},$ there exists $T=T(s,\mathcal{B}_{\gamma})>0$ such that for all $t\geq T,$ the strong solution $\mathrm{U}(\cdot)$ of system \eqref{LM_EQ} satisfies  
		\begin{align}\label{LM}
		\E\left[\|\mathrm{U}(s,s-t,\mathrm{U}_0)\|^p_{\X}\right] \leq 1+ C_{p,\lambda,\gamma,\theta_1,\theta_2,\theta_3}+\widetilde{C}_{p,\lambda,\gamma,\theta_1}\left[e^{-\theta_1\lambda^{\frac{\gamma}{2}} s}\int_{-\infty}^{s}e^{\theta_1\lambda^{\frac{\gamma}{2}} \tau}\|\g(\tau)\|^{\frac{p\gamma}{2(\gamma-1)}}_{\mathbb{Y}'}\d \tau\right],
		\end{align} for some $p\geq\delta+2$, where $\mathrm{U}_0\in\mathcal{B}_{\gamma}(s-t)$, $C_{p,\lambda,\gamma,\theta_1,\theta_2,\theta_3}$ is a constant depands on $p,\lambda,\gamma,\theta_1,\theta_2$ and $\theta_3$ only, and $\widetilde{C}_{p,\lambda,\gamma,\theta_1}$ is a constant depends on $p,\lambda,\gamma\text{ and }\theta_1$ only.
	\end{lemma}
	\begin{proof}
		By the infinite dimensional It\^o formula, we  find 
		\begin{align}\label{LM1}
		\d\|\U(\xi)\|^p_{\X}&=\frac{p}{2}\|\U(\xi)\|^{p-2}_{\X}\big(2\langle\F(\U(\xi))+\g(\xi),\U(\xi)\rangle+\varepsilon^2\|\upsigma(\xi,\U(\xi))\|^2_{\mathcal{L}_2(\X_0,\X)}\big)\d\xi \nonumber\\&\quad+ p\varepsilon\|\U(\xi)\|^{p-2}_{\X}\big(\U(\xi),\upsigma(\xi,\U(\xi))\d\W(\xi)\big)\nonumber\\&\quad+\frac{p(p-2)\varepsilon^2}{2} \|\U(\xi)\|^{p-4}_{\X}\|\upsigma^*(\xi,\U(\xi))\U(\xi)\|^2_{\X_0}\d\xi,
		\end{align}
		where $\upsigma^*$ is the adjoint of the operator  $\upsigma$. Taking the expectation of \eqref{LM1} and using the fact that $\int_{s-t}^{\xi}\|\U(\tau)\|^{p-2}_{\X}\big(\U(\tau),\upsigma(\tau,\U(\tau))\d\W(\tau)\big)$ is a local martingale, we obtain
		\begin{align}\label{LM2}
		&	\frac{\d}{\d\xi}\E\left[\|\mathrm{U}(\xi,s-t,\mathrm{U}_0)\|^p_{\X}\right]\nonumber\\&= \E\bigg[p\|\mathrm{U}(\xi,s-t,\mathrm{U}_0)\|^{p-2}_{\X}\langle\F(\mathrm{U}(\xi,s-t,\mathrm{U}_0))+\g(\xi),\mathrm{U}(\xi,s-t,\mathrm{U}_0)\rangle\nonumber\\&\quad+\frac{p\varepsilon^2}{2}\|\mathrm{U}(\xi,s-t,\mathrm{U}_0)\|^{p-2}_{\X}\|\upsigma(\xi,\mathrm{U}(\xi,s-t,\mathrm{U}_0))\|^2_{\mathcal{L}_2(\X_0,\X)}\nonumber\\&\quad+\frac{p(p-2)\varepsilon^2}{2}\|\mathrm{U}(\xi,s-t,\mathrm{U}_0)\|^{p-4}_{\X}\|\upsigma^*(\xi,\mathrm{U}(\xi,s-t,\mathrm{U}_0))\mathrm{U}(\xi,s-t,\mathrm{U}_0)\|^2_{\X_0}\bigg],
		\end{align}
		for a.e. $\xi\geq s-t.$ By the condition (F.3) of Hypothesis \ref{LM_F_H}, \eqref{poinLM} and Young's inequality, we estimate $p\|\U(\xi)\|^{p-2}_{\X}\langle\F(\U(\xi)),\U(\xi)\rangle$ as 
		\begin{align}\label{LM3}
		& p\|\U(\xi)\|^{p-2}_{\X}\langle\F(\U(\xi)),\U(\xi)\rangle\nonumber\\&\leq -p\theta_1\|\U(\xi)\|^{p-2}_{\X}\|\U(\xi)\|^{\gamma}_{\mathbb{Y}} +p\theta_2\|\U(\xi)\|^{p}_{\X} +p\theta_3\|\U(\xi)\|^{p-2}_{\X}\nonumber\\&\leq -(p-1)\theta_1\lambda^{\frac{\gamma}{2}}\|\U(\xi)\|^{p+\gamma-2}_{\X} -\theta_1\|\U(\xi)\|^{p-2}_{\X}\|\U(\xi)\|^{\gamma}_{\mathbb{Y}}+p\theta_2\|\U(\xi)\|^{p}_{\X} +p\theta_3\|\U(\xi)\|^{p-2}_{\X}\nonumber\\&\leq-p(\theta_1\lambda^{\frac{\gamma}{2}}+\theta_2+\theta_3)\|\U(\xi)\|^{p}_{\X} + \theta_1\lambda^{\frac{\gamma}{2}}(\gamma-2)\bigg[\frac{(p-1)\theta_1\lambda^{\frac{\gamma}{2}}+p\theta_2+p\theta_3}{\theta_1\lambda^{\frac{\gamma}{2}}(p+\gamma-2)}\bigg]^{\frac{p+\gamma-2}{\gamma-2}} \nonumber\\&\quad-\theta_1\|\U(\xi)\|^{p-2}_{\X}\|\U(\xi)\|^{\gamma}_{\mathbb{Y}}+p(\theta_2+\theta_3)\|\U(\xi)\|^{p}_{\X} +2\theta_3\bigg(\frac{p-2}{p}\bigg)^{\frac{p-2}{2}},
		\end{align}
		for a.e. $\xi\geq s-t.$ Using H\"older's and Young's inequalities, we also have 
		\begin{align}\label{LM4}
		p\|\U(\xi)\|^{p-2}_{\X}|\langle \g(\xi),\U(\xi)\rangle|&\leq p\|\U(\xi)\|^{p-2}_{\X}\|\g(\xi)\|_{\mathbb{Y}'}\|\U(\xi)\|_{\mathbb{Y}}\nonumber\\&\leq \theta_1\|\U(\xi)\|^{p-2}_{\X}\|\U(\xi)\|^{\gamma}_{\mathbb{Y}}+\theta_1\lambda^{\frac{\gamma}{2}}\|\U(\xi)\|^p_{\X}\nonumber\\&\quad+p^{\frac{p\gamma}{2(\gamma-1)}}\frac{2}{p}\bigg[\frac{p-2}{\theta_1\lambda^{\frac{\gamma}{2}}p}\bigg]^{\frac{p-2}{2}}\bigg[\frac{\gamma-1}{\gamma}\bigg(\frac{1}{\theta_1\gamma}\bigg)^{\frac{1}{\gamma-1}}\bigg]^{\frac{p}{2}}\|\g(\xi)\|^{\frac{p\gamma}{2(\gamma-1)}}_{\mathbb{Y}'}.
		\end{align}
		Let us choose $\varepsilon_0=\sqrt{\frac{(p-2)\theta_1\lambda^{\frac{\gamma}{2}}}{p(p-1)K}}$. Using the condition (M.2) of Hypothesis \ref{LMsigmaH}, for $0<\varepsilon\leq\varepsilon_0,$ we obtain
		\begin{align}\label{LM5}
		\frac{p\varepsilon^2}{2}\|\U(\xi)\|^{p-2}_{\X}\|\upsigma(\xi,\U(\xi))\|^2_{\mathcal{L}_2(\X_0,\X)}&\leq\frac{pK\varepsilon^2}{2}(\|\U(\xi)\|^{p-2}_{\X}+\|\U(\xi)\|^{p}_{\X}) \nonumber\\&\leq \frac{pK\varepsilon^2}{2} \bigg[\frac{2}{p}\bigg(\frac{p-2}{p}\bigg)^{\frac{p-2}{2}}+2\|\U(\xi)\|^{p}_{\X} \bigg]\nonumber\\&\leq \frac{\theta_1\lambda^{\frac{\gamma}{2}}}{p-1}\bigg(\frac{p-2}{p}\bigg)^{\frac{p}{2}} +\frac{(p-2)\theta_1\lambda^{\frac{\gamma}{2}}}{p-1}\|\U(\xi)\|^{p}_{\X}.
		\end{align}
		Using \eqref{LM5}, we find 
		\begin{align}\label{LM6}
		\frac{p(p-2)\varepsilon^2}{2}\|\U(\xi)\|^{p-4}_{\X}\|\upsigma^*(\xi,\U(\xi))\U(\xi)\|^2_{\X_0}&\leq\frac{p(p-2)\varepsilon^2}{2}\|\U(\xi)\|^{p-2}_{\X}\|\upsigma^*(\xi,\U(\xi))\|^2_{\mathcal{L}_2(\X,\X_0)}\nonumber\\&\leq(p-2) \theta_1\lambda^{\frac{\gamma}{2}}\bigg(\frac{p-2}{p}\bigg)^{\frac{p}{2}} +\frac{(p-2)^2\theta_1\lambda^{\frac{\gamma}{2}}}{p-1}	\|\U(\xi)\|^{p}_{\X},
		\end{align}
		for all $0<\varepsilon\leq\varepsilon_0.$ Combining  \eqref{LM3}-\eqref{LM6} and then substituting it in \eqref{LM2}, we deduce that 
		\begin{align}\label{LM7}
		&	\frac{\d}{\d\xi}\E\left[\|\mathrm{U}(\xi,s-t,\mathrm{U}_0)\|^p_{\X}\right]+\theta_1\lambda^{\frac{\gamma}{2}}\E\|\mathrm{U}(\xi,s-t,\mathrm{U}_0)\|^p_{\X} \leq C_{p,\lambda,\gamma,\theta_1,\theta_2,\theta_3}
		+\widetilde{C}_{p,\lambda,\gamma,\theta_1}\|\g(\xi)\|^{\frac{p\gamma}{2(\gamma-1)}}_{\mathbb{Y}'}.
		\end{align}
		for all $0<\varepsilon\leq\varepsilon_0$ and a.e. $\xi\geq s-t$. Multiplying \eqref{LM7} by $e^{\theta_1\lambda^{\frac{\gamma}{2}}\xi}$ and then integrating on $(s-t,s)$ with $t\geq0$, we obtain
		\begin{align}\label{LM8}
		\quad\E\left[\|\mathrm{U}(s,s-t,\mathrm{U}_0)\|^p_{\X}\right]& \leq e^{-\theta_1\lambda^{\frac{\gamma}{2}} t}\E\left[\|\mathrm{U}_0\|^p_{\X}\right] + C_{p,\lambda,\gamma,\theta_1,\theta_2,\theta_3}\nonumber\\&\quad+\widetilde{C}_{p,\lambda,\gamma,\theta_1}\left[e^{-\theta_1\lambda^{\frac{\gamma}{2}} s}\int_{-\infty}^{s}e^{\theta_1\lambda^{\frac{\gamma}{2}} \tau}\|\g(\tau)\|^{\frac{p\gamma}{2(\gamma-1)}}_{\mathbb{Y}'}\d \tau\right]. 
		\end{align}
		Since $\mathrm{U}_0\in \mathcal{B}_{\gamma}(s-t)$ and $\mathcal{B}_{\gamma}\in \mathfrak{D}_{\gamma}$, it can be easily seen that 
		\begin{align}\label{LM9}
		e^{-\theta_1\lambda^{\frac{\gamma}{2}} t} \E\left[\|\mathrm{U}_0\|^p_{\X}\right]&=e^{-\theta_1\lambda^{\frac{\gamma}{2}} s}e^{\theta_1\lambda^{\frac{\gamma}{2}}(s-t)} \E\left[\|\mathrm{U}_0\|^p_{\X}\right]\leq e^{-\theta_1\lambda^{\frac{\gamma}{2}} s}e^{\theta_1\lambda^{\frac{\gamma}{2}}(s-t)} \|\mathcal{B}_{\gamma}(s-t)\|^p_{\mathrm{L}^p(\Omega,\mathscr{F}_{s-t};\X)}\nonumber\\&\to0 \text{ as } t\to \infty.
		\end{align}
		Therefore, there exists $T=T(s,\mathcal{B}_{\gamma})>0$ such that  $e^{-\theta_1\lambda^{\frac{\gamma}{2}} t} \E\left[\|\mathrm{U}_0\|^p_{\X}\right]\leq 1,$ for all $t\geq T$. Using it in \eqref{LM8}, we arrive at 
		\begin{align}\label{LM10}
		\E\left[\|\mathrm{U}(s,s-t,\mathrm{U}_0)\|^p_{\X}\right] \leq 1+ C_{p,\lambda,\gamma,\theta_1,\theta_2,\theta_3}+\widetilde{C}_{p,\lambda,\gamma,\theta_1}\left[e^{-\theta_1\lambda^{\frac{\gamma}{2}} s}\int_{-\infty}^{s}e^{\theta_1\lambda^{\frac{\gamma}{2}} \tau}\|\g(\tau)\|^{\frac{p\gamma}{2(\gamma-1)}}_{\mathbb{Y}'}\d \tau\right],
		\end{align}
		for $t\geq T$, which completes the proof.
	\end{proof}
	Next, we prove the main result of this section, that is, the existence of weak $\mathfrak{D}_{\gamma}$-pullback mean random attractors for $\Phi,$ when $\gamma>2$.
	\begin{theorem}\label{WPMRA3}
		Suppose that all the conditions of Hypothesis \ref{LM_F_H}-\ref{LMf_H} are satisfied with $\gamma>2$. Then, there exists $\varepsilon_0>0$ such that for every $0<\varepsilon\leq\varepsilon_0$, the mean random dynamical system $\Phi$ for the system \eqref{LM_EQ} has a unique weak $\mathfrak{D}_{\gamma}$-pullback mean random attractor $\mathcal{G}=\{\mathcal{G}(s):s\in\R\}\in \mathfrak{D}_{\gamma}$ in $\mathrm{L}^p(\Omega,\mathscr{F};\X)$ over $(\Omega,\mathscr{F},\{\mathscr{F}_t\}_{t\in\R},\mathbb{P}).$  
	\end{theorem}
	\begin{proof}
		For a given $s\in\R$, let us define
		\begin{align*}
		\mathcal{J}(s) :=\{\mathrm{U}\in\mathrm{L}^p(\Omega,\mathscr{F}_s;\X):\E\left[\|\mathrm{U}\|^p_{\X}\right]\leq I(s)\},
		\end{align*}
		where $$I(s)=1+ C_{p,\lambda,\gamma,\theta_1,\theta_2,\theta_3}+\widetilde{C}_{p,\lambda,\gamma,\theta_1}\left[e^{-\theta_1\lambda^{\frac{\gamma}{2}} s}\int_{-\infty}^{s}e^{\theta_1\lambda^{\frac{\gamma}{2}} \tau}\|\g(\tau)\|^{\frac{p\gamma}{2(\gamma-1)}}_{\mathbb{Y}'}\d \tau\right].$$
		Since $\mathcal{J}(s)$ is a bounded closed convex subset of the reflexive Banach space $\mathrm{L}^p(\Omega,\mathscr{F}_s;\X)$, we infer that $\mathcal{J}(s)$ is weakly compact in $\mathrm{L}^p(\Omega,\mathscr{F}_s;\X)$. By Hypothesis \ref{LMf_H}, we also get
		\begin{align*}
		&	\lim_{s\to-\infty}e^{\theta_1\lambda^{\frac{\gamma}{2}} s}\|\mathcal{J}(s)\|^p_{\mathrm{L}^p(\Omega,\mathscr{F}_s;\X)}\\&=\lim_{s\to-\infty}e^{\theta_1\lambda^{\frac{\gamma}{2}} s}I(s)=\lim_{s\to-\infty}\bigg[\bigg(1+C_{p,\lambda,\gamma,\theta_1,\theta_2,\theta_3}\bigg)e^{\theta_1\lambda^{\frac{\gamma}{2}} s} +\widetilde{C}_{p,\lambda,\gamma,\theta_1}\int_{-\infty}^{s}e^{-\theta_1\lambda^{\frac{\gamma}{2}} \tau}\|\g(\tau)\|^{\frac{p\gamma}{2(\gamma-1)}}_{\mathbb{Y}'}\d \tau\bigg]\\&=0,
		\end{align*}
		which implies that $\mathcal{J}=\{\mathcal{J}(s):s\in\R\}\in\mathfrak{D}_{\gamma}$. By Lemma \ref{absorb_LM}, we infer that $\mathcal{J}$ is a weakly compact $\mathfrak{D}_{\gamma}$-pullback absorbing set for $\Phi$. By Theorem \ref{Main-T}, we immediately conclude that there exists a unique weak $\mathfrak{D}_{\gamma}$-pullback mean random attractor $\mathcal{G}\in\mathfrak{D}_{\gamma}$ of $\Phi$.
	\end{proof}
	Let us now consider the case $\gamma=2$. Proof of the following lemma is similar to Lemma \ref{absorb_LM}, except for the estimates \eqref{LM3}-\eqref{LM7}. 
	\begin{lemma}\label{absorb_LM_1}
		Let all the conditions of Hypothesis \ref{LM_F_H}-\ref{LMf_H} are satisfied with $\gamma=2$ and $\widetilde{\theta}=\theta_1-\frac{\theta_2+\theta_3}{\lambda}>0$ or $\frac{\theta_2+\theta_3}{\lambda}<\theta_1$. Then, there exists $\varepsilon_0>0$ such that for every $0<\varepsilon\leq\varepsilon_0$ and for every $s\in\R$ and $\mathcal{B}_2=\{\mathcal{B}_2(t)\}_{t\in\R}\in \mathfrak{D}_2,$ there exists $T=T(s,\mathcal{B}_2)>0$ such that for all $t\geq T,$ the solution $\mathrm{U}$ of system \eqref{LM_EQ} satisfies  
		\begin{align}\label{LM_1}
		\E\left[\|\mathrm{U}(s,s-t,\mathrm{U}_0)\|^p_{\X}\right] \leq 1+ C_{p,\lambda,\theta_1,\theta_2,\theta_3}+\widetilde{C}_{p,\lambda,\theta_1,\theta_2,\theta_3}\left[e^{-\widetilde{\theta}\lambda s}\int_{-\infty}^{s}e^{\widetilde{\theta}\lambda \tau}\|\g(\tau)\|^{p}_{\mathbb{Y}'}\d \tau\right],
		\end{align} for some $p\geq\delta+2$, where $\mathrm{U}_0\in\mathcal{B}_{\gamma}(s-t)$, $C_{p,\lambda,\theta_1,\theta_2,\theta_3}$ and $\widetilde{C}_{p,\lambda,\theta_1,\theta_2,\theta_3}$ are constants depend on $p,\lambda,\theta_1,\theta_2$ and $\theta_3$ only.
	\end{lemma}
	\begin{proof}
		\iffalse 
		By It\^o formula, we get
		\begin{align}\label{LM1_1}
		\d\|\U(\xi)\|^p_{\X}&=\frac{p}{2}\|\U(\xi)\|^{p-2}_{\X}\bigg(2\langle\F(\U(\xi))+\g(\xi),\U(\xi)\rangle+\varepsilon^2\|\upsigma(\xi,\U(\xi))\|^2_{\mathcal{L}_2(\X_0,\X)}\bigg)\d\xi \nonumber\\&\quad+ p\varepsilon\|\U(\xi)\|^{p-2}_{\X}\big(\U(\xi),\upsigma(\xi,\U(\xi))\d\W(\xi)\big)+\frac{p(p-2)\varepsilon^2}{2} \|\U(\xi)\|^{p-4}_{\X}\|\upsigma^*(\xi,\U(\xi))\U(\xi)\|^2_{\X_0}\d\xi,
		\end{align}
		where $\upsigma^*$ is the adjoint operator of $\upsigma$. Taking the expectation of \eqref{LM1_1}, we obtain, for $\xi\geq s-t,$
		\begin{align}\label{LM2_1}
		&\quad	\frac{\d}{\d\xi}\E\left[\|\mathrm{U}(\xi,s-t,\mathrm{U}_0)\|^p_{\X}\right]\nonumber\\&= \E\bigg[p\|\mathrm{U}(\xi,s-t,\mathrm{U}_0)\|^{p-2}_{\X}\langle\F(\mathrm{U}(\xi,s-t,\mathrm{U}_0))+\g(\xi),\mathrm{U}(\xi,s-t,\mathrm{U}_0)\rangle\nonumber\\&\quad+\frac{p\varepsilon^2}{2}\|\mathrm{U}(\xi,s-t,\mathrm{U}_0)\|^{p-2}_{\X}\|\upsigma(\xi,\mathrm{U}(\xi,s-t,\mathrm{U}_0))\|^2_{\mathcal{L}_2(\X_0,\X)}\nonumber\\&\quad+\frac{p(p-2)\varepsilon^2}{2}\|\mathrm{U}(\xi,s-t,\mathrm{U}_0)\|^{p-4}_{\X}\|\upsigma^*(\xi,\mathrm{U}(\xi,s-t,\mathrm{U}_0))\mathrm{U}(\xi,s-t,\mathrm{U}_0)\|^2_{\X_0}\bigg].
		\end{align}
		\fi 
		By the condition (F.3) of Hypothesis \ref{LM_F_H}, \eqref{poinLM} and Young's inequality, we get
		\begin{align}\label{LM3_1}
		&	p\|\U(\xi)\|^{p-2}_{\X}\langle\F(\U(\xi)),\U(\xi)\rangle\nonumber\\&\leq-p\theta_1\|\U(\xi)\|^{p-2}_{\X}\|\U(\xi)\|^{2}_{\mathbb{Y}}+p(\theta_2+\theta_3)\|\U(\xi)\|^{p}_{\X} +2\theta_3\bigg(\frac{p-2}{p}\bigg)^{\frac{p-2}{2}}\nonumber\\&\leq-p\left(\theta_1-\frac{\theta_2+\theta_3}{\lambda}\right)\|\U(\xi)\|^{p-2}_{\X}\|\U(\xi)\|^{2}_{\mathbb{Y}} +2\theta_3\bigg(\frac{p-2}{p}\bigg)^{\frac{p-2}{2}}\nonumber\\&\leq-\left(p-\frac{1}{2}\right)\lambda\widetilde{\theta}\|\U(\xi)\|^{p}_{\X}- \frac{\widetilde{\theta}}{2}\|\U(\xi)\|^{p-2}_{\X}\|\U(\xi)\|^{2}_{\mathbb{Y}} +2\theta_3\bigg(\frac{p-2}{p}\bigg)^{\frac{p-2}{2}},
		\end{align}
		for a.e. $\xi\geq s-t.$ 	A calculation similar to \eqref{LM4} gives 
		\begin{align}\label{LM4_1}
		p\|\U(\xi)\|^{p-2}_{\X}\langle \g(\xi),\U(\xi)\rangle&\leq p\|\U(\xi)\|^{p-2}_{\X}\|\g(\xi)\|_{\mathbb{Y}'}\|\U(\xi)\|_{\mathbb{Y}}\nonumber\\&\leq \frac{\widetilde{\theta}}{2}\|\U(\xi)\|^{p-2}_{\X}\|\U(\xi)\|^{2}_{\mathbb{Y}}+\frac{\widetilde{\theta}\lambda}{2}\|\U(\xi)\|^p_{\X}+\bigg[\frac{p-2}{\widetilde{\theta}\lambda}\bigg]^{\frac{p-2}{2}}\bigg[\frac{p}{\widetilde{\theta}}\bigg]^{\frac{p}{2}}\|\g(\xi)\|^{p}_{\mathbb{Y}'}.
		\end{align}
		Choosing $\varepsilon_0=\sqrt{\frac{(p-2)\widetilde{\theta}\lambda}{p(p-1)K}}$, using Hypothesis \ref{LMsigmaH} (M.2) and Young's inequality,  for $0<\varepsilon\leq\varepsilon_0,$ we obtain
		\begin{align}\label{LM5_1}
		\frac{p\varepsilon^2}{2}\|\U(\xi)\|^{p-2}_{\X}\|\upsigma(\xi,\U(\xi))\|^2_{\mathcal{L}_2(\X_0,\X)}&\leq \frac{\widetilde{\theta}\lambda}{p-1}\bigg(\frac{p-2}{p}\bigg)^{\frac{p}{2}} +\frac{(p-2)\widetilde{\theta}\lambda}{p-1}\|\U(\xi)\|^{p}_{\X}.
		\end{align}
		Using \eqref{LM5_1},  for all $0<\varepsilon\leq\varepsilon_0,$ we find
		\begin{align}\label{LM6_1}
		\frac{p(p-2)\varepsilon^2}{2}\|\U(\xi)\|^{p-4}_{\X}\|\upsigma^*(\xi,\U(\xi))\U(\xi)\|^2_{\X_0}\leq\frac{(p-2) \widetilde{\theta}\lambda}{p-1}\bigg(\frac{p-2}{p}\bigg)^{\frac{p}{2}} +\frac{(p-2)^2\widetilde{\theta}\lambda}{p-1}	\|\U(\xi)\|^{p}_{\X}.
		\end{align}
		Combining \eqref{LM3_1}-\eqref{LM6_1} and then substituting it in \eqref{LM2}, we obtain
		\begin{align}\label{LM7_1}
		&	\frac{\d}{\d\xi}\E\left[\|\mathrm{U}(\xi,s-t,\mathrm{U}_0)\|^p_{\X}\right]+\widetilde{\theta}\lambda\E\|\mathrm{U}(\xi,s-t,\mathrm{U}_0)\|^p_{\X} \leq C_{p,\lambda,\theta_1,\theta_2,\theta_3}
		+\widetilde{C}_{p,\lambda,\theta_1,\theta_2,\theta_3}\|\g(\xi)\|^{p}_{\mathbb{Y}'},
		\end{align}
		for all $0<\varepsilon\leq\varepsilon_0$ and a.e. $\xi\geq s-t$.	Multiplying \eqref{LM7_1} by $e^{\widetilde{\theta}\lambda\xi}$ and then integrating on $(s-t,s)$ with $t\geq0$, we deduce that 
		\begin{align}\label{LM8_1}
		\E\left[\|\mathrm{U}(s,s-t,\mathrm{U}_0)\|^p_{\X}\right]& \leq e^{-\widetilde{\theta}\lambda t}\E\left[\|\mathrm{U}_0\|^p_{\X}\right] + C_{p,\lambda,\theta_1,\theta_2,\theta_3}\nonumber\\&\quad+\widetilde{C}_{p,\lambda,\theta_1,\theta_2,\theta_3}\left[e^{-\widetilde{\theta}\lambda s}\int_{-\infty}^{s}e^{\widetilde{\theta}\lambda \tau}\|\g(\tau)\|^{p}_{\mathbb{Y}'}\d \tau\right]. 
		\end{align}
		Since $\mathrm{U}_0\in \mathcal{B}_2(s-t)$ and $\mathcal{B}_2\in \mathfrak{D}_2$, we get
		\begin{align}\label{LM9_1}
		e^{-\widetilde{\theta}\lambda t} \E\left[\|\mathrm{U}_0\|^p_{\X}\right]&=e^{-\widetilde{\theta}\lambda s}e^{\widetilde{\theta}\lambda(s-t)} \E\left[\|\mathrm{U}_0\|^p_{\X}\right]\leq e^{-\widetilde{\theta}\lambda s}e^{\widetilde{\theta}\lambda(s-t)} \|\mathcal{B}_2(s-t)\|^p_{\mathrm{L}^p(\Omega,\mathscr{F}_{s-t};\X)}\nonumber\\&\to0 \text{ as } t\to \infty.
		\end{align}
		Therefore, there exists $T=T(s,\mathcal{B}_2)>0$ such that  $e^{-\widetilde{\theta}\lambda t} \E\left[\|\mathrm{U}_0\|^p_{\X}\right]\leq 1,$ for all $t\geq T$. By \eqref{LM8_1}, we finally obtain, for $t\geq T$,
		\begin{align}\label{LM10_1}
		\E\left[\|\mathrm{U}(s,s-t,\mathrm{U}_0)\|^p_{\X}\right] \leq 1+ C_{p,\lambda,\theta_1,\theta_2,\theta_3}+\widetilde{C}_{p,\lambda,\theta_1,\theta_2,\theta_3}\left[e^{-\widetilde{\theta}\lambda s}\int_{-\infty}^{s}e^{\widetilde{\theta}\lambda \tau}\|\g(\tau)\|^{p}_{\mathbb{Y}'}\d \tau\right],
		\end{align}which completes the proof.
	\end{proof}
	Next, one can establish the existence of weak $\mathfrak{D}_2$-pullback mean random attractors for $\Phi$ when $\gamma=2$ and $\widetilde{\theta}>0$ similarly as in the proof of Theorem \ref{WPMRA3}.
	\begin{theorem}\label{WPMRA4}
		Suppose that all the conditions of Hypothesis \ref{LM_F_H}-\ref{LMf_H} are satisfied with $\gamma=2$ and $\widetilde{\theta}:=\theta_1-\frac{\theta_2+\theta_3}{\lambda}>0$ or $\frac{\theta_2+\theta_3}{\lambda}<\theta_1$. Then, there exists $\varepsilon_0>0$ such that for every $0<\varepsilon\leq\varepsilon_0$, the mean random dynamical system $\Phi$ for the system \eqref{LM_EQ} has a unique weak $\mathfrak{D}_2$-pullback mean random attractor $\widetilde{\mathcal{G}}=\{\widetilde{\mathcal{G}}(s):s\in\R\}\in \mathfrak{D}_2$ in $\mathrm{L}^p(\Omega,\mathscr{F};\X)$ over $(\Omega,\mathscr{F},\{\mathscr{F}_t\}_{t\in\R},\mathbb{P}).$  
	\end{theorem}
	\subsection{Applications}
	In this subsection, we provide some examples which comes under the functional framework discussed in Theorems \ref{WPMRA3} and \ref{WPMRA4}. All the following examples satisfies conditions (F.1)-(F.4) of Hypothesis \ref{LM_F_H}. The interested readers are referred to see section 6, \cite{GLS} for more details. 
	\begin{itemize}
		\item [1.] Stochastic Burgers type and reaction diffusion (semilinear stochastic) equations (Example 6.1, \cite{GLS}).
		\item[2.] Stochastic 2D Navier-Stokes equations (Example 6.3, \cite{GLS}) and hydrodynamic models like  stochastic magnetohydrodynamic (MHD) equations (Subsection 2.1.2, \cite{GLS}),  stochastic Boussinesq model for the B\'enard convection (Subsection 2.1.3, \cite{GLS}),  stochastic 2D magnetic B\'enard problem, stochastic 3D Leray-$\alpha$ model (Example 6.5, \cite{GLS}), stochastic shell model of turbulence (Subsection 2.1.6, \cite{GLS}).
		\item [3.] Stochastic power law fluids (Example 6.8, \cite{GLS}).
		\item [4.] Stochastic Ladyzhenskaya model (Example 6.9, \cite{GLS}).
		\item [5.] SPDE with monotone coefficients (Subsection 6.8, \cite{GLS}).
	\end{itemize}
	
	\medskip\noindent
	{\bf Acknowledgments:}    The first author would like to thank the Council of Scientific $\&$ Industrial Research (CSIR), India for financial assistance (File No. 09/143(0938)/2019-EMR-I).  M. T. Mohan would  like to thank the Department of Science and Technology (DST), Govt of India for Innovation in Science Pursuit for Inspired Research (INSPIRE) Faculty Award (IFA17-MA110).

\end{document}